\documentclass[reqno,10pt,centertags]{amsart}
\usepackage{amsmath,amsthm,amscd,amssymb,latexsym,upref,enumerate}
\usepackage{bbm}
\usepackage{nicefrac}
\usepackage{hyperref}

\newcommand*{\mailto}[1]{\href{mailto:#1}{\nolinkurl{#1}}}
\newcommand{\arxiv}[1]{\href{http://arxiv.org/abs/#1}{arXiv:#1}} 
 
\newcommand{\bbC}{{\mathbb{C}}}

\newcommand{\bbN}{{\mathbb{N}}}

\newcommand{\bbR}{{\mathbb{R}}}
\newcommand{\bbT}{{\mathbb{T}}}

\newcommand{\cD}{{\mathcal D}}

\newcommand{\cH}{{\mathcal H}}

\newcommand{\e}{{\varepsilon}}

\DeclareMathOperator{\supp}{supp}
\DeclareMathOperator{\rank}{rank}
\DeclareMathOperator{\ran}{ran}
\DeclareMathOperator{\dom}{dom}

\DeclareMathOperator*{\slim}{s-lim}

\renewcommand{\Re}{\text{\rm Re}}
\renewcommand{\Im}{\text{\rm Im}}

\newcommand{\loc}{\text{\rm{loc}}}

\newcommand{\beq}{\begin{equation}}
\newcommand{\enq}{\end{equation}}

\newcommand{\no}{\notag}
\newcommand{\lb}{\label}
\newcommand{\f}{\frac}

\newcommand{\ol}{\overline}

\newcommand{\wti}{\widetilde}
\newcommand{\Oh}{O}
\newcommand{\oh}{o}
\newcommand{\hatt}{\widehat}
\newcommand{\bi}{\bibitem}
\newcommand{\slimes}{\text{\rm{l.i.m.}}}

\let\geq\geqslant
\let\leq\leqslant

\makeatletter
\def\theequation{\@arabic\c@equation}

\allowdisplaybreaks 
\numberwithin{equation}{section}

\newtheorem{theorem}{Theorem}[section]

\newtheorem{lemma}[theorem]{Lemma}
\newtheorem{corollary}[theorem]{Corollary}
\newtheorem{hypothesis}[theorem]{Hypothesis}

\theoremstyle{remark}
\newtheorem{remark}[theorem]{Remark}

\begin{document}

\title[Supersymmetry and Weyl--Titchmarsh Theory]{Supersymmetry and 
Schr\"odinger-Type Operators with Distributional Matrix-Valued Potentials}

\author[J.\ Eckhardt]{Jonathan Eckhardt}
\address{Faculty of Mathematics\\ University of Vienna\\
Oskar-Morgenstern-Platz 1\\ 1090 Wien\\ Austria}
\email{\mailto{jonathan.eckhardt@univie.ac.at}}
\urladdr{\url{http://homepage.univie.ac.at/jonathan.eckhardt/}}

\author[F.\ Gesztesy]{Fritz Gesztesy}
\address{Department of Mathematics,
University of Missouri,
Columbia, MO 65211, USA}
\email{\mailto{gesztesyf@missouri.edu}}
\urladdr{\url{http://www.math.missouri.edu/personnel/faculty/gesztesyf.html}}

\author[R.\ Nichols]{Roger Nichols}
\address{Mathematics Department, The University of Tennessee at Chattanooga, 415 EMCS Building, Dept. 6956, 615 McCallie Ave, Chattanooga, TN 37403, USA}
\email{\mailto{Roger-Nichols@utc.edu}}
\urladdr{\url{http://www.utc.edu/faculty/roger-nichols/}} 

\author[G.\ Teschl]{Gerald Teschl}
\address{Faculty of Mathematics\\ University of Vienna\\
Oskar-Morgenstern-Platz 1\\ 1090 Wien\\ Austria\\ and International
Erwin Schr\"odinger
Institute for Mathematical Physics\\ Boltzmanngasse 9\\ 1090 Wien\\ Austria}
\email{\mailto{Gerald.Teschl@univie.ac.at}}
\urladdr{\url{http://www.mat.univie.ac.at/~gerald/}} 

\thanks{Research supported by the Austrian Science Fund (FWF) under Grant No.\ Y330.}
\thanks{J. Spectral Theory {\bf 4}, 715--768 (2014)} 
\date{\today}

\begin{abstract} 
Building on work on Miura's transformation by Kappeler, Perry, Shubin, and Topalov, 
we develop a detailed spectral theoretic treatment of Schr\"odinger operators with matrix-valued 
potentials, with special emphasis on distributional potential coefficients. 

Our principal method relies on a supersymmetric (factorization) formalism underlying Miura's transformation, which intimately connects the triple of operators $(D, H_1, H_2)$ of the form 
\[
D= \left(\begin{smallmatrix} 0 & A^* \\ A & 0 \end{smallmatrix}\right)  
\, \text{ in } \, L^2(\mathbb{R})^{2m} \, \text{ and } \, H_1 = A^* A, \;\;
H_2 = A A^* \, \text{ in } L^2(\mathbb{R})^m.
\]
Here $A= I_m (d/dx) + \phi$ in $L^2(\mathbb{R})^m$, with a 
matrix-valued coefficient $\phi = \phi^* \in L^1_{\text{loc}}(\mathbb{R})^{m \times m}$, $m \in \mathbb{N}$, 
thus explicitly permitting distributional potential coefficients $V_j$ in $H_j$, $j=1,2$, 
where
\[
H_j = - I_m \frac{d^2}{dx^2} + V_j(x), \quad V_j(x) = \phi(x)^2 + (-1)^{j} \phi'(x), \; j=1,2.  
\]
Upon developing Weyl--Titchmarsh theory for these generalized Schr\"odinger operators $H_j$, 
with (possibly, distributional) matrix-valued potentials $V_j$, we provide some spectral 
theoretic applications, including a derivation of the corresponding spectral representations 
for $H_j$, $j=1,2$. Finally, we derive a local Borg--Marchenko uniqueness theorem for $H_j$, 
$j=1,2$, by employing the underlying supersymmetric structure and reducing it to the known 
local Borg--Marchenko uniqueness theorem for $D$. 
\end{abstract}

\maketitle

\section{Introduction}  \lb{s1}

This paper was inspired by an investigation concerning ``the Miura map on the line'' by 
Kappeler, Perry, Shubin, and Topalov \cite{KPST05} in 2005. In it, the authors consider the well-known 
Miura map, 
\begin{equation} 
\phi \mapsto \phi^2 + \phi',     \lb{1.1}
\end{equation} 
which relates appropriate classes of solutions of the 
Korteweg--de Vries (KdV) and modified Korteweg--de Vries (mKdV) equation (cf., e.g., 
\cite{Ge91}, \cite{Ge92}, \cite{GSS91}, \cite{GS90}  
and the literature cited therein). The Miura map is closely related with factorizations of the KdV Lax 
operator $H$, the one-dimensional Schr\"odinger operator in $L^2(\bbR)$, into a product of two first-order operators 
of the form
\begin{equation}
H = A A^* = - \f{d^2}{dx^2} + V(x), \quad V(x) = \phi(x)^2 +  \phi'(x),   \lb{1.2} 
\end{equation}
where
\begin{equation}
A = \f{d}{dx} + \phi(x), \quad A^* = - \f{d}{dx} + \phi(x).     \lb{1.3} 
\end{equation}
(While these factorizations are formal at this point, their precise mathematical content is discussed in 
Sections \ref{s2} and \ref{s3}.) 
In particular, under the assumption $\phi \in L^2_{\loc}(\bbR)$, $\phi$ real-valued a.e.\ on $\bbR$, 
this permits the authors in \cite{KPST05} to discuss real-valued distributional potentials 
$V \in H^{-1}_{\loc}(\bbR)$, and hence 
accomplish a remarkable extension of the standard theory of self-adjoint one-dimensional Schr\"odinger operators in $L^2(\bbR)$ which typically deals with the case of real-valued potentials 
$V \in L^1_{\loc}(\bbR)$ (resp., $V \in L^2_{\loc}(\bbR)$). This program is carried out in \cite{KPST05} by relying on oscillation theoretic techniques and Hartman's concept of principal and nonprincipal solutions. In particular, the principal focus of \cite{KPST05} is a detailed investigation of the Miura map \eqref{1.1}, its range, and its geometry on the real line, with special emphasis on function spaces with low regularity.

As it happens, the Miura map \eqref{1.1} is intimately connected with an underlying supersymmetric 
structure which relates a triple of operators $(D,H_1,H_2)$ of the form,
\begin{align}
\begin{split} 
& D= \begin{pmatrix} 0 & A^* \\ A & 0 \end{pmatrix} \, \text{ in } \, L^2(\bbR)^{2},    \lb{1.4} \\
& H_1 = A^* A \, \text{ and } \, H_2 = A A^* \, \text{ in } \, L^2(\bbR). 
\end{split} 
\end{align}
Most notably in this context, spectral properties of one of $D, H_1, H_2$ essentially determine the corresponding spectral properties of the remaining two operators in the triple $(D,H_1,H_2)$ (as 
described in Appendix \ref{sA}). In particular, since, in accordance with \eqref{1.3}, 
$A = (d/dx) + \phi(x)$, $A^* = - (d/dx) + \phi(x)$ deal with the (non-distributional) 
coefficient $\phi \in L^2_{\loc}(\bbR)$ only, so does the 
Dirac-type operator $D$. Consequently, spectral theory (including Weyl--Titchmarsh theory) for the 
standard Dirac operator $D$ should lead in an effective and streamlined manner to spectral and 
Weyl--Titchmarsh theory for the generalized Schr\"odinger operators $H_1$ and $H_2$ which may harbor distributional potential coefficients $V_j$, as
\begin{equation}
H_j = - \f{d^2}{dx^2} + V_j(x), \quad V_j(x) = \phi(x)^2 + (-1)^{j} \phi'(x), \; j=1,2.     \lb{1.5} 
\end{equation}
Realizing this circle of ideas is precisely what is offered in this paper. Moreover, the fact that exploiting the underlying supersymmetric structure is most natural in this context will become clear as we can effortless incorporate two important generalizations as follows:  \\[2mm] 
$\bullet$ We permit more general coefficients $\phi$ and hence (distributional) coefficients $V_j$, 
\hspace*{2.5mm} $j=1,2$, as we only need to assume $\phi \in L^1_{\loc}(\bbR)$.  \\[2mm]
$\bullet$ We actually consider the matrix-valued case in which $\phi$ and $V_j$, $j=1,2$, are 
\hspace*{2.5mm} $m \times m$ self-adjoint matrices a.e.\ on $\bbR$.

\medskip 

Before describing the content of this paper, it is appropriate to comment on the history of singular 
Sturm--Liouville operator with special emphasis on the papers devoted to distributional potentials. 

The particular case of point interactions as special distributional coefficients in Schr\"odinger operators received enormous attention, too numerous to be mentioned here in detail. Hence, we only refer to the standard monographs by Albeverio, Gesztesy, H{\o}egh-Krohn, and Holden \cite{AGHKH05} and  
Albeverio and Kurasov \cite{AK01}, and some of the more recent developments in 
Albeverio, Kostenko, and Malamud \cite{AKM10}, Kostenko and Malamud \cite{KM10}, \cite{KM10a}.
We also mention the case of discontinuous Schr\"odinger operators originally considered by Hald \cite{Ha84},
motivated by the inverse problem for the torsional modes of the earth. For recent development in this direction we refer to Shahriari, Jodayree Akbarfam, and Teschl \cite{SJT12}. 

The case of Schr\"odinger operators with strongly singular and oscillating potentials that should be 
mentioned in this context goes back to studies by Baeteman and Chadan \cite{BC75},  \cite{BC76}, Combescure \cite{Co80}, Combescure and Ginibre \cite{CG76}, Pearson \cite{Pe79}, Rofe-Beketov and Hristov \cite{RH66}, \cite{RH69} and a more recent contribution treating distributional potentials by 
Herczy\'nski \cite{He89}. The case of very general (i.e., three-coefficient) singular Sturm--Liouville operators including distributional potentials has been studied by Bennewitz and Everitt \cite{BE83} in 1983 (see also \cite[Sect.\ I.2]{EM99}). They restrict their considerations to compact intervals and focus on the special case of a left-definite setting. An extremely thorough and systematic investigation, including even and odd higher-order operators defined in terms of appropriate quasi-derivatives, and in the general case of matrix-valued coefficients (including distributional potential coefficients in the context of Schr\"odinger-type operators) was presented by Weidmann \cite{We87} in 1987. In fact, the general approach in \cite{BE83} and \cite{We87} draws on earlier discussions of quasi-derivatives in Shin \cite{Sh38}--\cite{Sh43}, 
Naimark \cite[Ch.\ V]{Na68}, and Zettl \cite{Ze75}. 
Still, it appears that the distributional coefficients treated in \cite{BE83} did not catch on and subsequent authors referring to this paper mostly focused on the various left and right-definite aspects developed therein. Similarly, it seems likely that the extraordinary generality exerted by Weidmann \cite{We87} in his treatment of higher-order differential operators obscured the fact that he already dealt with distributional potential coefficients back in 1987.   

However, it was not until 1999 that Savchuk and Shkalikov \cite{SS99} started a new development for Sturm--Liouville (resp., Schr\"odinger) operators with distributional potential coefficients in connection with 
areas such as, self-adjointness proofs, spectral and inverse spectral theory, oscillation properties, spectral properties in the non-self-adjoint context, etc. In addition to the important series of papers by Savchuk and Shkalikov \cite{SS99}--\cite{SS10}, we mention other groups such as Albeverio, Hryniv, and Mykytyuk \cite{AHM08}, Bak and Shkalikov \cite{BS02},  Ben Amara and Shkalikov \cite{BS09}, Ben Amor and Remling \cite{BR05}, Davies \cite{Da13}, 
Djakov and Mityagin \cite{DM09}--\cite{DM12}, Eckhardt and Teschl \cite{ET13}, 
Frayer, Hryniv, Mykytyuk, and Perry \cite{FHMP09}, Gesztesy and Weikard \cite{GW13}, 
Goriunov and Mikhailets \cite{GM10}, \cite{GM10a}, Hryniv \cite{Hr11}, 
Kappeler and M\"ohr \cite{KM01}, Kappeler, Perry, Shubin, and Topalov \cite{KPST05}, 
Kappeler and Topalov \cite{KT04}, 
Hryniv and Mykytyuk \cite{HM01}--\cite{HM12}, Hryniv, Mykytyuk, and Perry \cite{HMP11}--\cite{HMP11a}, 
Kato \cite{Ka10}, Korotyaev \cite{Ko03}, \cite{Ko12}, Maz'ya and Shaposhnikova \cite[Ch.\ 11]{MS09}, 
Maz'ya and Verbitsky \cite{MV02}--\cite{MV06}, Mikhailets and Molyboga \cite{MM04}--\cite{MM09}, 
Mirzoev and Safanova \cite{MS11}, Mykytyuk and Trush \cite{MT10}, 
Sadovnichaya \cite{Sa10}, \cite{Sa11}. In particular, the paper by Mirzoev and Safanova \cite{MS11} is closely related to the present one as it also employs the use of a quasi-derivative of the type 
$f^{[1]}=f' + \phi f$ to define a Schr\"odinger-type operator via a Miura-type transformation and appears 
to be the only paper known to us since Weidmann's 1987 monograph and the very recent 
\cite{MM13} that deals with the matrix-valued 
case, that is, $f$ is $\bbC^m$-valued, $\phi$ is $\bbC^{m \times m}$-valued, $m\in\bbN$. The prime focus of \cite{MS11} is the computation of deficiency indices of the underlying minimal operator.

It should be mentioned that some of the attraction in connection with distributional potential coefficients 
in the Schr\"odinger operator clearly stems from the low-regularity investigations of solutions of the 
Korteweg--de Vries (KdV) equation. We mention, for instance, Buckmaster and Koch \cite{BK12}, 
Grudsky and Rybkin \cite{GR12}, Kappeler and M\"ohr \cite{KM01}, Kappeler and Topalov \cite{KT05}, 
\cite{KT06}, and Rybkin \cite{Ry10}. 

The case of strongly singular potentials at an endpoint and the associated Weyl--Titchmarsh--Kodaira theory 
for Schr\"odinger operators can already be found in the seminal paper by Kodaira \cite{Ko49}. A gap in
Kodaira's approach was later circumvented by Kac \cite{Ka67}. The theory did not receive much further attention
until it was independently rediscovered and further developed by Gesztesy and Zinchenko \cite{GZ06}.
This soon lead to a systematic development of Weyl--Titchmarsh theory for strongly singular potentials and we mention, for instance, Eckhardt \cite{E14}, Eckhardt and Teschl \cite{ET13a}, Fulton \cite{Fu08}, 
Fulton and Langer \cite{FL10}, Fulton, Langer, and Luger \cite{FLL12}, 
Kostenko, Sakhnovich, and Teschl \cite{KST10}, \cite{KST12}, \cite{KST12a}, \cite{KT11}, \cite{KT12}, and 
Kurasov and Luger \cite{KL11}. 

We also mention that a different approach to general (i.e., three-coefficient) singular 
Sturm--Liouville operators (which are not necessarily assumed to be bounded from below) on an arbitrary 
interval $(a,b) \subseteq \bbR$, has been developed simultaneously in \cite{EGNT13} in the special 
scalar case $m=1$. This paper systematically develops Weyl--Titchmarsh theory for differential expressions of the type 
\begin{equation}
  \tau f = \frac{1}{r} \left( - \big(p[f' + \phi f]\big)' + \phi p[f' + \phi f] + qf\right)   \lb{1.6}
\end{equation}
and hence is very close in spirit to the general discussion provided by Weidmann \cite{We87}. Here the coefficients $p, q, r, \phi$ are real-valued and Lebesgue measurable on $(a,b)$, with $p\not=0$, $r>0$ 
a.e.\ on $(a,b)$, and $p^{-1}, q, r, \phi \in L^1_{\loc}((a,b); dx)$, and $f$ is supposed to satisfy
\begin{equation} 
f \in AC_{\text{loc}}((a,b)), \; p[f' + \phi f] \in AC_{\text{loc}}((a,b)),     \lb{1.7} 
\end{equation}
with $AC_{\loc}((a,b))$ denoting the set of locally absolutely continuous functions on $(a,b)$. In particular, this study includes distributional coefficients. (The paper \cite{EGNT13} does not employ the supersymmetric formalism.) Inverse spectral theory for these operators is treated in \cite{EGNT13a}.

It remains to briefly describe the content of this paper: Section \ref{s2} recalls the basics of 
Weyl--Titchmarsh theory for supersymmetric Dirac-type operators 
$D= \left(\begin{smallmatrix} 0 & A^* \\ A & 0 \end{smallmatrix}\right)$ 
in $L^2(\bbR)^{2m}$, where $A = I_m (d/dx) + \phi$ in $L^2(\bbR)^m$ (cf.\ \eqref{2.3}--\eqref{2.6}), with a 
matrix-valued coefficient $\phi = \phi^* \in L^1_{\loc}(\bbR)^{m \times m}$, $m \in \bbN$, following the treatment in \cite{CG02}. In particular, we review Weyl--Titchmarsh theory for $D$ on the half-line and the full real line, including the $2m \times 2m$ matrix-valued Green's function of $D$. In Section \ref{s3} we exploit the supersymmetric structure of $D$ and analyze the underlying generalized Schr\"odinger operators 
$H_1 = A^* A$ and $H_2 = A A^*$ in $L^2(\bbR)^m$. We derive the Weyl--Titchmarsh solutions for 
$H_j$, $j=1,2$, given those of $D$ described in Section \ref{s2}, and describe the precise connection between the half-line Weyl--Titchmarsh matrices of $H_j$, $j=1,2$, and $D$. In addition, we construct the 
$m \times m$ matrix-valued Green's functions of $H_j$, $j=1,2$, and the corresponding analogs 
belonging to the half-lines $[x_0,\infty)$ and $(-\infty,x_0]$ with a Dirichlet boundary condition at $x_0$. In our final Section \ref{s4} we provide some spectral theoretic applications of the supersymmetric approach outlined in Section \ref{s3} and after deriving the fundamental aspects of Weyl--Titchmarsh theory for the generalized Schr\"odinger operators $H_j$, $j=1,2$, and a discussion of the corresponding spectral representations, we derive a local Borg--Marchenko uniqueness theorem by utilizing the known analog for the Dirac operator $D$. Supersymmetric Dirac-type operators and associated commutation methods are briefly summarized in Appendix \ref{sA}.

Next, we briefly summarize some of the notation used in this paper: 
All $m\times m$ matrices $M\in\bbC^{m\times m}$ will be  considered over
the field of complex numbers $\bbC$. Moreover, $I_m$ denotes
the identity matrix in $\bbC^{m\times m}$ for $m\in\bbN$, $M^*$
the adjoint (i.e., complex conjugate transpose), $M^\top$ the transpose 
of the matrix $M$, $[M_1,M_2] = M_1 M_2 - M_2 M_1$ denotes the standard commutator 
of two matrices $M_j\in \bbC^{m \times m}$, $j=1,2$.

$AC([a,b])$ (resp., $AC_{\loc}(c,d)))$ denotes the
set of (locally) absolutely continuous functions on $[a,b]$ (resp., $(c,d)$).
We also agree that $L^2((a,b))^m$, $m\in\bbN$, without explicit depiction of a measure always 
denotes $L^2((a,b); dx)^m$, with $dx$ representing the Lebesgue measure restricted to 
$(a,b)$, $-\infty \leq a < b \leq \infty$, in particular,
\begin{equation}
L^2((a,b))^m = \bigg\{U:(a,b)\to\bbC^{m} \, \bigg| \, \int_a^b dx \,
\|U(x)\|^2_{\bbC^{m}}<\infty \bigg\},  \quad m \in \bbN.    \lb{1.8}
\end{equation}
The identity operator in $L^2((a,b))^m$ will simply be denoted by $I$. 

For ease of notation we will typically use the short cut $[x_0,\pm\infty)$ to denote the 
half-lines $[x_0,\infty)$ or $(-\infty,x_0]$ for some $x_0 \in \bbR$. 

Finally, let $T$ be a linear operator mapping (a subspace of) a
Hilbert space into another, with $\dom(T)$, $\ran(T)$, and $\ker(T)$ denoting the
domain, range, and kernel (i.e., null space) of $T$. 
The closure of a closable operator $S$ is denoted by $\ol S$. 
The spectrum and resolvent set of a closed linear operator in a Hilbert space will be denoted by 
$\sigma(\cdot)$ and $\rho(\cdot)$, respectively.

\section{Weyl--Titchmarsh Matrices for Supersymmetric Dirac Operators}  \lb{s2}

In this preparatory section we briefly review the Weyl--Titchmarsh theory for Dirac-type operators $D$ in the special supersymmetric case. In particular,  $D$ is constructed as a special case of the theory of singular Hamiltonian systems as pioneered by Hinton and Shaw \cite{HS81}--\cite{HS86} (see also \cite{HS93}, \cite{HS97}) and applied to Dirac-type operators in 
\cite{CG02}. 

Throughout this section we closely follow the treatment in \cite{CG02} (simplified to the present supersymmetric Dirac-type operator) and hence are making the following assumptions.

\begin{hypothesis} \lb{h2.1}
Suppose $\phi \in L^1_{\loc}(\bbR)^{m \times m}$, $m \in \bbN$, and 
$\phi(\cdot) = \phi(\cdot)^*$ a.e.\ on $\bbR$. 
\end{hypothesis}

Given Hypothesis \ref{h2.1} we introduce the maximally defined operators 
$A$ and $A^+$ in $L^2(\bbR)^m$ by
\begin{align}
\begin{split}
& (A u)(x) = u'(x) + \phi (x) u(x)  \text{ for a.e.\ $x \in \bbR$,}  \\
& \, u \in \dom(A) = 
\big\{v \in L^2(\bbR)^m \,\big|\, v \in AC_{\loc}(\bbR); \, (v' + \phi v) 
\in L^2(\bbR)^m \big\},    \lb{2.1} 
\end{split} 
\end{align}    
and 
\begin{align}
\begin{split}
& (A^+ u)(x) = - u'(x) + \phi (x) u(x)   \text{ for a.e.\ $x \in \bbR$,} \\
& \, u \in \dom(A^+) = 
\big\{v \in L^2(\bbR)^m \,\big|\, v \in AC_{\loc}(\bbR); \, (v' - \phi v) 
\in L^2(\bbR)^m\big\}.     \lb{2.2} 
\end{split} 
\end{align}   
In addition, we consider the maximally defined Dirac-type operator $D$ in 
$L^2(\bbR)^{2m}$ by
\begin{align}
& (D U)(x) =  \bigg(\begin{pmatrix} 0 & A^+ \\ A & 0 \end{pmatrix}
\begin{pmatrix} u_1 \\ u_2 \end{pmatrix}\bigg)(x)  
= \begin{pmatrix} (A^+ u_2)(x) \\ (A u_1)(x) \end{pmatrix}
   \text{ for a.e.\ $x \in \bbR$,}    \no \\ 
& \, U = \begin{pmatrix} u_1 \\ u_2 \end{pmatrix} \in \dom(D) = 
\dom(A) \oplus \dom(A^+)    \lb{2.3} \\
& \hspace*{2.5cm} = \bigg\{V = \begin{pmatrix} v_1 \\ v_2 \end{pmatrix}  \in 
L^2(\bbR)^{2m} \, \bigg | \, V \in AC_{\loc}(\bbR)^{2m}; \, 
DV \in L^2(\bbR)^{2m}\bigg\}.    \no 
\end{align} 

The basic known result on $A$, $A^+$, and $D$ then reads as follows:

\begin{theorem} [\cite{CG02}, \cite{HS81}, \cite{HS83}, \cite{HS84}] \lb{t2.2}
Assume Hypothesis \ref{h2.1}. Then $A$ and $A^+$ are closed in 
$L^2(\bbR)^m$ and 
\begin{equation}
A^* = A^+, \quad (A^+)^* = A.     \lb{2.4}
\end{equation}
In addition, $D$ is self-adjoint in $L^2(\bbR)^{2m}$, that is, $D$ is of the form, 
\begin{equation}
D = \begin{pmatrix} 0 & A^* \\ A & 0 \end{pmatrix}.    \lb{2.5}
\end{equation}
\end{theorem} 
\begin{proof} 
By \cite[Lemma\ 2.15]{CG02}, the differential expression 
\begin{equation}
\cD = \begin{pmatrix} 0 & -I_m(d/dx) + \phi(x) \\ I_m(d/dx) + \phi(x) & 0 
\end{pmatrix} \text{ for a.e.\ $x \in \bbR$,}     \lb{2.6}
\end{equation}
is in the limit point case at $\pm \infty$. (For a subsequent and more general result we refer to \cite{LM03}, see also \cite{LM00} and \cite{LO82} for such proofs under stronger hypotheses on $\phi$). Combining this result with the Weyl--Titchmarsh theory developed for singular Hamiltonian systems by Hinton and Shaw in a series of papers \cite{HS81}, \cite{HS83}, \cite{HS84}, 
yields self-adjointness of the maximal operator associated to the differential expression $\cD$. By \eqref{A.2}, $A$ and $A^+$ are hence necessarily closed, and consequently, adjoint to each other, proving \eqref{2.4} and \eqref{2.5}. 
\end{proof}

Because of the special structure \eqref{2.5}, $D$ is called a {\it supersymmetric} Dirac-type operator. For a discussion of its general properties we refer to Appendix \ref{sA}.

Because of \eqref{2.4}, we identify $A^+$ and $A^*$ from this point on. 
  
In order to discuss $m \times m$ Weyl--Titchmarsh matrices corresponding to $D$ on the half-lines $(-\infty,x_0]$ and $[x_0,\infty)$, we introduce boundary condition parameters 
$\alpha = (\alpha_1 \; \alpha_2) \in\bbC^{m\times 2m}$ satisfying the conditions 
\begin{equation}
\alpha\alpha^*=I_m, \quad \alpha J\alpha^*=0, \, \text{  where } \, 
J = \begin{pmatrix} 0 & -I_m \\ I_m & 0 \end{pmatrix}.      \lb{2.7}
\end{equation} 
Explicitly, this reads
\begin{equation}
\alpha_1\alpha_1^* +\alpha_2\alpha_2^*=I_m, \quad \alpha_2\alpha_1^* 
-\alpha_1\alpha_2^*=0. \lb{2.8}
\end{equation}
In fact, one also has
\begin{equation}
\alpha_1^*\alpha_1 +\alpha_2^*\alpha_2=I_m, \quad \alpha_2^*\alpha_1 
-\alpha_1^*\alpha_2=0, \lb{2.9}
\end{equation}
as is clear from
\begin{equation}
\begin{pmatrix} \alpha_1 & \alpha_2\\ -\alpha_2 & \alpha_1 \end{pmatrix}
\begin{pmatrix} \alpha_1^* & -\alpha_2^*\\ \alpha_2^* & \alpha_1^*
\end{pmatrix}=I_{2m}=\begin{pmatrix} \alpha_1^* & -\alpha_2^*\\ 
\alpha_2^* & \alpha_1^* \end{pmatrix}\begin{pmatrix} \alpha_1 & \alpha_2\\
-\alpha_2 & \alpha_1 \end{pmatrix}, \lb{2.10}
\end{equation}
since any left inverse matrix is also a right inverse, and vice versa. 
Moreover, from \eqref{2.9} one obtains 
\begin{equation}
\alpha^*\alpha J + J\alpha^*\alpha =  J.    \lb{2.11}
\end{equation}
The particular choice where $\alpha$ equals 
\begin{equation}
\alpha_0=(I_m\; 0) 
\end{equation} 
will play a fundamental role later on.

Next, denote by $U_{\pm}(\zeta,\,\cdot\,,x_0,\alpha)$ the $2m \times m$ 
matrix-valued Weyl--Titchmarsh solutions associated with $\cD U = \zeta U$, 
$\zeta\in\bbC\backslash\bbR$, defined by the property that the $m$ columns of $U_\pm$ span the deficiency spaces $N(\zeta,\pm\infty)$, 
$\zeta \in \bbC\backslash\bbR$, given by
\begin{align}
& N(\zeta,\pm\infty)= \big\{V \in L^2((x_0,\pm\infty))^{2m} \,\big|\, 
V \in AC([x_0, x_0 \pm R])^{2m} \, \text{for all $R>0$};     \no \\ 
& \hspace*{6.3cm} \cD V = \zeta V \, \text{a.e.\ on $(x_0,\pm\infty)$} \big\}, 
\lb{2.12} 
\end{align}
and normalized such that 
\begin{align}
U_\pm(\zeta,x,x_0,\alpha)&=\begin{pmatrix}u_{\pm,1}(\zeta,x,x_0,\alpha) \\
u_{\pm,2}(\zeta,x,x_0,\alpha)  \end{pmatrix}
= \Psi(\zeta,x,x_0,\alpha)\begin{pmatrix} I_m \\
M^D_{\pm}(\zeta,x_0,\alpha) \end{pmatrix}  \no \\
&=\begin{pmatrix}\vartheta_1(\zeta,x,x_0,\alpha)
& \varphi_1(\zeta,x,x_0,\alpha)\\
\vartheta_2(\zeta,x,x_0,\alpha)
& \varphi_2(\zeta,x,x_0,\alpha)\end{pmatrix}
\begin{pmatrix} I_m \\
M^D_{\pm}(\zeta,x_0,\alpha) \end{pmatrix}.      \lb{2.13}
\end{align}
Here $M^D_{\pm}(\zeta,x_0,\alpha)$ represents an $m\times m$ matrix, and $\Psi(\zeta,x,x_0,\alpha)$, $\vartheta_j(\zeta,x,x_0,\alpha)$, and
$\varphi_j(\zeta,x,x_0,\alpha)$, $j=1,2$, are defined as follows: 
$\Psi(\zeta,x,x_0,\alpha)$ satisfies $\cD \Psi = \zeta \Psi$ a.e.\ on $\bbR$, normalized such that 
\begin{equation}\lb{2.14}
\Psi(\zeta,x_0,x_0,\alpha)=(\alpha^* \; J\alpha^*)=
\begin{pmatrix} \alpha_1^* & -\alpha_2^* \\
\alpha_2^* & \alpha_1^* \end{pmatrix}.
\end{equation}
Partitioning $\Psi(\zeta,x,x_0,\alpha)$ as follows,
\begin{equation}
\Psi(\zeta,x,x_0,\alpha) = 
\begin{pmatrix}\vartheta_1(\zeta,x,x_0,\alpha) &
\varphi_1(\zeta,x,x_0,\alpha)\\
\vartheta_2(\zeta,x,x_0,\alpha)& \varphi_2(\zeta,x,x_0,\alpha)
\end{pmatrix},\lb{2.15}
\end{equation}
defines $\vartheta_j(\zeta,x,x_0,\alpha)$ and $\varphi_j(\zeta,x,x_0,\alpha)$, 
$j=1,2$, as $m\times m$ matrices, entire with
respect to $\zeta\in\bbC$, and normalized according to \eqref{2.14}.

The matrices $M^D_{\pm}(\zeta,x_0,\alpha)$ represent the sought after half-line Weyl--Titchmarsh matrices associated with the Dirac-type operator 
$D$, whose basic properties can be summarized as follows:

\begin{theorem}
[\cite{AN76}, \cite{AD56}, \cite{Ca76}, \cite{CG02}, \cite{GT00}, \cite{HS81},
\cite{HS82}, \cite{HS86}, \cite{KS88}] \lb{t2.3} ${}$ \\ 
Suppose Hypothesis \ref{h2.1}, let $\zeta\in\bbC\backslash\bbR$,
$x_0\in\bbR$, and denote by $\alpha, \gamma\in\bbC^{m\times 2m}$ matrices satisfying \eqref{2.7}.\ Then the following hold: \\
$(i)$ $\pm M^D_{\pm}(\,\cdot\,,x_0,\alpha)$ is an $m\times m$ matrix-valued Nevanlinna--Herglotz function of maximal rank $m$. In particular,
\begin{align}
& \Im(\pm M^D_{\pm}(\zeta,x_0,\alpha)) \geq 0, \quad \zeta \in \bbC_+, \\
& M^D_{\pm}(\overline \zeta,x_0,\alpha)=M^D_{\pm}(\zeta,x_0,\alpha)^*, 
\lb{2.17} \\
& \rank (M^D_{\pm}(\zeta,x_0,\alpha))=m,  \lb{2.18}\\
& \lim_{\varepsilon\downarrow 0} M^D_{\pm}(\nu+
i\varepsilon,x_0,\alpha) \text{ exists for a.e.\ $\nu\in\bbR$},\lb{2.19}\\
& M^D_{\pm}(\zeta,x_0,\alpha) = [-\alpha J \gamma^* +
\alpha\gamma^* M^D_{\pm}(\zeta,x_0,\gamma)] [ \alpha\gamma^*
+  \alpha J \gamma^*M^D_{\pm}(\zeta,x_0,\gamma)]^{-1}.   \lb{2.20} 
\end{align}
Local singularities of $\pm M^D_{\pm}(\,\cdot\,,x_0,\alpha)$ and 
$\mp M^D_{\pm}(\,\cdot\,,x_0,\alpha)^{-1}$ are necessarily real and at most of first
order in the sense that 
\begin{align}
&\mp \lim_{\epsilon\downarrow0}
\left(i\epsilon\,
M^D_{\pm}(\nu+i\epsilon,x_0,\alpha)\right) \geq 0, \quad 
\pm \lim_{\epsilon\downarrow0}
\big(i\epsilon \, M^D_{\pm}(\nu+i\epsilon,x_0,\alpha)^{-1}\big)
\geq 0, \quad \nu \in \bbR. 
\lb{2.21}  
\end{align}
$(ii)$  $\pm M^D_{\pm}(\,\cdot\,,x_0,\alpha)$ admits the representation 
\begin{equation}
\pm M^D_{\pm}(\zeta,x_0,\alpha)=F_\pm(x_0,\alpha)+\int_\bbR
d\Omega_\pm^D (\nu,x_0,\alpha) \,
\big[(\nu - \zeta)^{-1}-\nu(1+\nu^2)^{-1}\big], \lb{2.22} 
\end{equation}
where
\begin{equation}
F_\pm(x_0,\alpha)=F_\pm(x_0,\alpha)^*, \quad \int_\bbR
\big\|d\Omega^D_{\pm} (\nu,x_0,\alpha)\big\|_{\bbC^{m\times m}} \,
(1+\nu^2)^{-1}<\infty. 
\end{equation}
Moreover,
\begin{equation}
\Omega_{\pm}^D((\mu,\nu],x_0,\alpha)
=\lim_{\delta\downarrow
0}\lim_{\varepsilon\downarrow 0}\f1\pi
\int_{\mu+\delta}^{\nu+\delta} d\nu' \, \Im\big(\pm
M^D_{\pm}(\nu'+i\varepsilon,x_0,\alpha)\big).
\end{equation}
$(iii)$ $\Im\big(M^D_{\pm}(\,\cdot\,,x_0,\alpha)\big)$ satisfies  
\begin{align}
\Im\big(M^D_{\pm}(\zeta,x_0,\alpha)\big) &= \Im(\zeta) \int_{x_0}^{\pm\infty} dx \,
U_\pm(\zeta,x,x_0,\alpha)^* U_\pm(\zeta,x,x_0,\alpha)   \no \\
&= \Im(\zeta) \int_{x_0}^{\pm\infty} dx \, \big[
u_{\pm,1}(\zeta,x,x_0,\alpha)^* u_{\pm,1}(\zeta,x,x_0,\alpha)   \lb{2.25} \\
& \hspace*{2.85cm} +
u_{\pm,2}(\zeta,x,x_0,\alpha)^* u_{\pm,2}(\zeta,x,x_0,\alpha) \big].   \no 
\end{align}
\end{theorem}

For completeness we also recall that the $2m\times 2m$ Green's matrix (i.e., the integral kernel of the resolvent) of $D$ is given in terms of $U_\pm$ and $M_{\pm}$ by
\begin{align}
& G^D(\zeta,x,x^\prime) = (D - \zeta I)^{-1}(x,x')   \no \\
& \quad = U_\mp(\zeta,x,x_0,\alpha) \big[M^D_-(\zeta,x_0,\alpha) 
-M^D_+(\zeta,x_0,\alpha)\big]^{-1} U_\pm(\overline \zeta,x^\prime,x_0,\alpha)^*, 
\lb{2.26} \\
& \hspace*{6.4cm}  x\lessgtr x^\prime, \; x, x' \in\bbR, \; 
\zeta\in\bbC\backslash\bbR.  \no 
\end{align}
Of course, $G^D(\zeta,x,x^\prime)$ is independent of the choice of reference point $x_0 \in \bbR$, and independent of the boundary condition parameter $\alpha$ satisfying \eqref{2.7} used in 
$M^D_{\pm}(\zeta,x_0,\alpha)$ and $U_\pm(\zeta,\cdot\,,x_0,\alpha)$. One also notes that \eqref{2.26} extends as usual to all 
$\zeta \in \rho(D)$. In the particular case $\alpha_0=(I_m\; 0)$ one obtains
\begin{equation}
U_\pm(\zeta,x_0,x_0,\alpha_0) 
= \begin{pmatrix}u_{\pm,1}(\zeta,x_0,x_0,\alpha_0) \\
u_{\pm,2}(\zeta,x_0,x_0,\alpha_0)  \end{pmatrix}
= \begin{pmatrix} I_m \\ M^D_{\pm}(\zeta,x_0,\alpha_0) \end{pmatrix}.  
\lb{2.27}
\end{equation}

The self-adjoint half-line Dirac operators $D_{\pm} (\alpha)$ in 
$L^2([x_0,\pm\infty))^{2m}$ associated with a
self-adjoint boundary condition at $x_0$ indexed by
$\alpha\in\bbC^{m\times 2m}$ satisfying \eqref{2.7}, are of the form 
\begin{align}
& (D_{\pm} (\alpha) U)(x) = (\cD U)(x)  \text{ for a.e.\ $x \in [x_0,\pm\infty)$,}     
\no \\
& \, U \in \dom(D_{\pm} (\alpha)) = \big\{V \in L^2([x_0, \pm \infty))^{2m} 
\, \big| \, 
V \in AC([x_0, x_0 \pm R])^{2m}     \lb{2.28} \\
& \hspace*{3.3cm} \text{ for all $R > 0$}; \, \alpha\phi(x_0)=0; \, \cD V \in 
L^2([x_0, \pm \infty))^{2m} \big\}.    \no
\end{align}
The $m\times m$ matrix-valued spectral function of $D_{\pm} (\alpha)$  then generates the 
measure $\Omega^D_\pm(\,\cdot\,,x_0,\alpha)$ in \eqref{2.22}.

We conclude this section with a brief description of the full-line 
$2m\times 2m$ Weyl--Titchmarsh matrix $\mathbf{M}^D(\zeta,x_0,\alpha)$ associated with $D$ as described 
in \cite{HS81}--\cite{HS86}: 
\begin{align}
\mathbf{M}^D (\zeta,x_0,\alpha)
&=\big(\mathbf{M}^D_{j,j^\prime}(\zeta,x_0,\alpha)\big)_{j,j^\prime=1,2},  \quad 
\zeta\in\bbC\backslash\bbR,    \no \\
\mathbf{M}^D_{0,0}(\zeta,x_0,\alpha)&=[M^D_-(\zeta,x_0,\alpha)-M^D_+(\zeta,x_0,\alpha)]^{-1},
\no \\ 
\mathbf{M}^D_{0,1}(\zeta,x_0,\alpha)&=2^{-1}
[M^D_-(\zeta,x_0,\alpha) - M^D_+(\zeta,x_0,\alpha)]^{-1}  \no \\
& \quad \times [M^D_-(\zeta,x_0,\alpha)+M^D_+(\zeta,x_0,\alpha)],     \lb{2.29} \\
\mathbf{M}^D_{1,0}(\zeta,x_0,\alpha)&=2^{-1} [M^D_-(\zeta,x_0,\alpha)+M^D_+(\zeta,x_0,\alpha)]  \no \\
& \quad \times [M^D_-(\zeta,x_0,\alpha) - M^D_+(\zeta,x_0,\alpha)]^{-1},\no \\
\mathbf{M}^D_{1,1}(\zeta,x_0,\alpha)&=M^D_\pm(\zeta,x_0,\alpha)
[M^D_-(\zeta,x_0,\alpha)-M^D_+(\zeta,x_0,\alpha)]^{-1}M^D_\mp(\zeta,x_0,\alpha). \no 
\end{align}

The basic results on $\mathbf{M}^D(\,\cdot\,,x_0,\alpha)$ then read as follows.

\begin{theorem} [\cite{GT00}, \cite{HS81}, \cite{HS82}, 
\cite{HS86}, \cite{KS88}] \lb{t2.4} 
Assume Hypothesis \ref{h2.1} and suppose that $\zeta\in\bbC 
\backslash \bbR$, $x_0\in\bbR$, and that $\alpha\in\bbC^{m\times 2m}$
satisfies \eqref{2.7}.  Then the following hold: \\
$(i)$ $\mathbf{M}^D(\,\cdot\,,x_0,\alpha)$ is a matrix-valued Nevanlinna--Herglotz function of maximal rank 
$2m$ with representation
\begin{equation}
\mathbf{M}^D(\zeta,x_0,\alpha)=\mathbf{F}(x_0,\alpha)+\int_\bbR d\mathbf{\Omega}^D (\nu,x_0,\alpha)\,
\big[(\nu - \zeta)^{-1}-\nu (1+\nu^2)^{-1}\big], \lb{2.30} 
\end{equation}
where
\begin{equation}
\mathbf{F}(x_0,\alpha)=\mathbf{F}(x_0,\alpha)^*, \quad \int_\bbR \big\| d\mathbf{\Omega}^D (\nu,x_0,\alpha)
\big\|_{\bbC^{2m\times 2m}} \,(1+\nu^2)^{-1}<\infty.   \lb{2.31} 
\end{equation}
Moreover,
\begin{equation}
\mathbf{\Omega}^D ((\mu,\nu],x_0,\alpha)=\lim_{\delta\downarrow
0}\lim_{\varepsilon\downarrow 0}\f1\pi
\int_{\mu+\delta}^{\nu+\delta} d\nu' \, 
\Im\big(\mathbf{M}^D(\nu' + i\varepsilon,x_0,\alpha)\big).     \lb{2.32} 
\end{equation}
$(ii)$ $\zeta\in\rho(D)$ if and only if $\mathbf{M}^D(\zeta,x_0,\alpha)$ is
holomorphic near $\zeta$. 
\end{theorem}

Finally, observe that supersymmetry implies various symmetries for the associated quantities.

\begin{lemma}
Assume Hypothesis \ref{h2.1}. The operators $D_\pm(\alpha_0)$ are supersymmetric and satisfy $\mathfrak{S}_3 D_\pm(\alpha_0) \mathfrak{S}_3 = -D_\pm(\alpha_0)$, where
\begin{equation}
\mathfrak{S}_3 = \begin{pmatrix} I_m & 0 \\ 0 & -I_m \end{pmatrix}.
\end{equation} 
Moreover,
\begin{align}
& \mathfrak{S}_3 \big(\varphi_1(\zeta,x,x_0,\alpha_0) \ \ \varphi_2(\zeta,x, x_0,\alpha_0)\big)^\top
 = - \big(\varphi_1(-\zeta,x, x_0,\alpha_0) \ \ \varphi_2(-\zeta,x, x_0,\alpha_0)\big)^\top, \no\\
& \mathfrak{S}_3 \big(\vartheta_1(\zeta,x, x_0,\alpha_0) \ \ \vartheta_2(\zeta,x, x_0,\alpha_0)\big)^\top
 = \big(\vartheta_1(-\zeta,x, x_0,\alpha_0) \ \ \vartheta_2(-\zeta,x, x_0,\alpha_0)\big)^\top, \no\\
& \mathfrak{S}_3 U_\pm(\zeta,x,x_0,\alpha_0) = U_\pm(-\zeta,x,x_0,\alpha_0),\\
& M_\pm^D(\zeta,x_0,\alpha_0) = - M_\pm^D(-\zeta,x_0,\alpha_0), \no\\
& d\Omega_\pm^D(\nu,x_0,\alpha_0)= d\Omega_\pm^D(-\nu,x_0,\alpha_0). \no
\end{align}
Similarly, $D$ is supersymmetric, $\mathfrak{S}_3 D \mathfrak{S}_3 = -D$, and
\begin{align}
\begin{split} 
& \mathbf{M}^D(\zeta,x_0,\alpha_0) = - \mathfrak{S}_3 \mathbf{M}^D(-\zeta,x_0,\alpha_0) \mathfrak{S}_3,  \\
& d\mathbf{\Omega}^D(\nu,x_0,\alpha_0)= \mathfrak{S}_3 d\mathbf{\Omega}^D(-\nu,x_0,\alpha_0) \mathfrak{S}_3. 
\end{split} 
\end{align}
\end{lemma}
\begin{proof}
Clearly $D_\pm(\alpha_0)$ are of the form \eqref{A.2} in this case and hence they are supersymmetric. Moreover,
the symmetries for the solutions follow since both sides satisfy the same differential equation and the same initial
conditions, respectively normalizations. The claims for $D$ are immediate from the ones for $D_\pm(\alpha_0)$.
\end{proof}

\section{Supersymmetry and the connection between the Weyl--Titchmarsh Matrices for Dirac and Generalized Schr\"odinger-Type Operators}  \lb{s3}

In our principal section we provide the connection with the matrix-valued Weyl--Titchmarsh functions of the supersymmetric Dirac-type operator $D$ described in Section \ref{s2} and two naturally associated generalized 
Schr\"odinger-type operators $H_j$, $j =1,2$, given by 
\begin{equation} 
D^2 = \begin{pmatrix} A^* A & 0 \\ 0 & A A^* \end{pmatrix} = 
H_1 \oplus H_2 \, \text{ in } \, 
L^2(\bbR)^{2m} \simeq L^2(\bbR)^{m} \oplus L^2(\bbR)^{m},   \lb{3.1}
\end{equation}
in particular, we denote 
\begin{equation}
H_1 = A^* A, \quad  H_2 = A A^*,      \lb{3.2}
\end{equation}
with $A$ and $A^*$ given by \eqref{2.1} and \eqref{2.2}, respectively (cf.\ 
Theorem \ref{t2.2}). 

While $D$ contains the locally integrable $m \times m$ matrix-valued coefficient 
$\phi$, the associated generalized Schr\"odinger operators $H_j$, $j=1,2$, will exhibit distributional potentials and hence are outside the standard Weyl--Titchmarsh theory for Sturm--Liouville operators with locally integrable $m \times m$ matrix-valued potentials. Our supersymmetric approach will enable us to make the transition from the usual $L^1_{\loc}$-potentials in Schr\"odinger operators to distributional 
$H^{-1}_{\loc}$-potentials (and more general situations) in an effortless manner, thereby underscoring the power of these supersymmetric arguments.

To describe $H_j$, $j=1,2$, in $L^2(\bbR)$ in detail, we first introduce the following two kinds of quasi-derivatives,
\begin{align}
u^{[1,1]} (x) &= (A u)(x) = u'(x) + \phi(x) u(x)  \text{ for a.e.\ $x \in \bbR$,} 
\quad u \in \dom(A),     \lb{3.3} \\
v^{[1,2]} (x) & = - (A^* v)(x) = v'(x) - \phi(x) v(x)  \text{ for a.e.\ 
$x \in \bbR$,} \quad v \in \dom(A^*).     \lb{3.4}    
\end{align}
Thus, one infers,
\begin{align}
& (H_1 u)(x) = (A^* A u)(x) = (\tau_1 u)(x) 
= - \big(u^{[1,1]}\big)'(x) + \phi(x) u^{[1,1]} (x) 
  \text{ for a.e.\ $x \in \bbR$,}    \no \\
& \, u \in \dom(H_1) = \big\{v \in L^2(\bbR)^m \, \big| \, v, v^{[1,1]} \in 
AC_{\loc}(\bbR)^m; \lb{3.5}\\
&\hspace*{3.65cm} \big[\big(v^{[1,1]}\big)' + \phi v^{[1,1]}\big] \in 
L^2(\bbR)^m\big\},\no
\end{align}
and
\begin{align} 
& (H_2 u)(x) = (A A^* u)(x) = (\tau_2 u)(x)
= - \big(u^{[1,2]}\big)'(x) - \phi(x) u^{[1,2]} (x) 
  \text{ for a.e.\ $x \in \bbR$,}    \no \\
& \, u \in \dom(H_2) = \big\{v \in L^2(\bbR)^m \, \big| \, v, v^{[1,2]} \in 
AC_{\loc}(\bbR)^m;\lb{3.6} \\
&\hspace*{3.65cm} \big[\big(v^{[1,2]}\big)' + \phi v^{[1,2]}\big] \in 
L^2(\bbR)^m\big\}.\no
\end{align}
  
Formally, $\tau_j$, $j=1,2$, are of the form
\begin{align}
\tau_j = - I_m \f{d^2}{dx^2} + V_j(x), \quad 
V_j(x) = \phi(x)^2 + (-1)^{j} \phi'(x), \quad j=1,2,
\end{align}  
but one notices that, in general, neither $\phi^2$ is locally integrable (unless one makes the stronger assumption $\phi \in L^2_{\loc} (\bbR)^{m \times m}$), nor is $\phi'$ a function (unless one assumes in addition that 
$\phi \in AC_{\loc} (\bbR)^{m \times m}$). 

By inspection, the second-order initial value problems, 
\begin{align}
\begin{split}
& ((\tau_j - z) f)(x) = g(x) \text{ for a.e.\ $x \in \bbR$,} \quad 
f, f^{[1,j]} \in AC_{\loc}(\bbR)^m, \; g \in L^1_{\loc} (\bbR)^m, \\
& \, f(x_0) = c_0, \;\, f^{[1,j]}(x_0) = d_0,\, j=1,2,    \lb{3.8}
\end{split} 
\end{align} 
for some $x_0 \in \bbR$, $c_0, d_0 \in \bbC$, 
are equivalent to the first-order initial value problems  
\begin{align}
& \begin{pmatrix} f(x) \\ f^{[1,j]}(x) \end{pmatrix}^{\prime} = 
\begin{pmatrix} (-1)^j \phi(x) & 1 \\ -z & (-1)^{j+1}\phi(x) \end{pmatrix} 
\begin{pmatrix} f(x) \\ f^{[1,j]}(x) \end{pmatrix} - 
\begin{pmatrix} 0 \\ g(x) \end{pmatrix}  \text{ for a.e.\ $x \in \bbR$,}  
\no \\
& \begin{pmatrix} f(x_0) \\ f^{[1,j]}(x_0) \end{pmatrix} 
= \begin{pmatrix} c_0 \\ d_0 \end{pmatrix},\quad j=1,2,  \lb{3.10}
\end{align}
respectively. Since by Hypothesis \ref{h2.1}, 
$\phi \in L^1_{\loc} (\bbR)^{m \times m}$, the initial value problems in
\eqref{3.10} (and hence those in \eqref{3.8}) are uniquely solvable by \cite[Theorem\ 16.1]{Na68} 
(see also \cite[Theorem\ 10.1]{EE89} and \cite[Theorem\ 16.2]{Na68}). 

Next, suppose that for some $1\leq p \leq m$, $U = (u_1 \,\, u_2)^\top$ is a distributional $2m \times p$ solution of $D U = \zeta U$, that is, 
\begin{align}
\begin{split}
& u_j \in AC_{\loc}(\bbR)^{m \times p}, \; j=1,2, \\
& u_1^{[1,1]} = A u_1 \in L^1_{\loc}(\bbR)^{m \times p},   \quad 
u_2^{[1,2]} = - A^* u_2  \in L^1_{\loc}(\bbR)^{m \times p}.    \lb{3.12}
\end{split} 
\end{align}
Then, if $\zeta \neq 0$, the supersymmetric structure of $D$ in 
\eqref{2.5} actually implies that also 
\begin{align}
& u_1^{[1,1]} = A u_1 = \zeta u_2 \in AC_{\loc}(\bbR)^{m \times p},   \lb{3.13} \\
& u_2^{[1,2]} = - A^* u_2 = - \zeta u_1 \in AC_{\loc}(\bbR)^{m \times p},  
\lb{3.14} 
\end{align}
and hence that $u_j$ are actually distributional $m \times p$ solutions of $H_j u = \zeta^2 u$, $j=1,2$, that is,
\begin{align}
\begin{split}
& u_j, u_j^{[1,j]} \in AC_{\loc}(\bbR)^{m \times p}, \quad 
\big(u_j^{[1,j]}\big)' \in L^1_{\loc}(\bbR)^{m \times p},  \\  
& \tau_j u_j = - \big(u_j^{[1,j]}\big)' + (-1)^{j+1}\phi u_j^{[1,j]} = \zeta^2 u_j,\quad j=1,2.  \lb{3.15}
\end{split}
\end{align}

Thus, applying the $L^2$-property \eqref{2.25} and 
\eqref{3.12}--\eqref{3.15} to the Weyl--Titchmarsh solutions 
$U_\pm(\zeta,\,\cdot\,,x_0,\alpha)$ associated with the Dirac-type operator $D$, then shows that $u_{\pm,j}(\zeta,\,\cdot\,,x_0,\alpha)$ are 
Weyl--Titchmarsh solutions associated with $H_j$, $j=1,2$, replacing the complex energy parameter $\zeta$ by $z = \zeta^2$. Moreover, 
introducing the following fundamental system 
$s_j (z,\,\cdot\,, x_0), c_j (z,\,\cdot\,, x_0)$, $j=1,2$,  
of $m \times m$ matrix solutions of $\tau_j u = z u$, $z \in \bbC$, $j=1,2$, normalized for arbitrary $z \in \bbC$ by 
\begin{align}
& s_j(z, x_0, x_0) = 0, & s_j^{[1,j]}(z, x_0, x_0) & = I_m,                          \lb{3.17}    \\
& c_j(z, x_0, x_0) = I_m, & c_j^{[1,j]}(z, x_0, x_0) & = 0, \qquad j=1,2,    \lb{3.18} 
\end{align}
one observes as usual that for fixed $x, x_0 \in \bbR$, 
$s_j (\,\cdot\,, x, x_0), c_j (\,\cdot\,, x, x_0)$ are entire. The connection with the solutions $\varphi_j$ and $\vartheta_j$, $j=1,2$, of $D$ is given by
\begin{align}
& s_1(z, x, x_0) = \zeta^{-1} \varphi_1(\zeta,x,x_0, \alpha_0), \quad c_1(z, x, x_0) 
= \vartheta_1(\zeta,x,x_0,\alpha_0), \\
& s_2(z, x, x_0) = \zeta^{-1} \vartheta_2(\zeta,x,x_0,\alpha_0), \quad c_2(z, x, x_0) 
= \varphi_2(\zeta,x,x_0,\alpha_0), \quad z=\zeta^2.
\end{align}
In addition, introducing the 
Weyl--Titchmarsh solutions $\psi_{\pm,j}(z,\, \cdot \, ,x_0)$ for $H_j$, $j=1,2$, via
\begin{align}
\psi_{\pm,1} (z,\,\cdot\,, x_0) &= u_{\pm,1} (\zeta,\,\cdot\,, x_0,\alpha_0), 
\lb{3.21} \\ 
\psi_{\pm,2} (z,\,\cdot\,, x_0) &= u_{\pm,2} (\zeta,\,\cdot\,, x_0,\alpha_0)
M^D_{\pm} (\zeta,x_0,\alpha_0)^{-1},   \lb{3.22} \\
& \hspace*{1.6cm}  z = \zeta^2, \; \zeta \in \bbC \backslash \bbR, \; 
j=1,2,   \no 
\end{align}
(the right-hand sides being independent of the choice of branch for $\zeta$) and the generalized 
Dirichlet-type $m \times m$ matrix-valued Weyl--Titchmarsh functions 
$\hatt M_{\pm,0,j} (\,\cdot\,,x_0)$ of $H_j$, 
\begin{align}
\hatt M_{\pm,0,1} (z,x_0) &= \zeta M^D_{\pm}(\zeta,x_0,\alpha_0),  \lb{3.23} \\
\hatt M_{\pm,0,2} (z,x_0) &= -\zeta M^D_{\pm}(\zeta,x_0,\alpha_0)^{-1}, 
\lb{3.24} \\
& \hspace*{8.7mm}  z = \zeta^2, \; \zeta \in \bbC\backslash\bbR,   \no 
\end{align}
one infers from \eqref{2.27} that
\begin{align}
\psi_{\pm,j} (z,\,\cdot\,, x_0) = c_j (z,\,\cdot\,,x_0) 
+ s_j(z,\,\cdot\,,x_0) \hatt M_{\pm,0,j} (z,x_0), \quad 
z \in \bbC\backslash [0,\infty), \; j=1,2.   \lb{3.25} 
\end{align}
Indeed, \eqref{3.25} follows from combining \eqref{2.27}, \eqref{3.13}, and 
\eqref{3.14} (for $p=m$), which in turn imply 
\begin{align}
 \psi_{\pm,j} (z, x_0, x_0) &= I_m,    \lb{3.26} \\
 \psi_{\pm,j}^{[1,j]} (z, x_0, x_0) &= \hatt M_{\pm,0,j} (z,x_0), \quad 
j=1,2   \lb{3.27}
\end{align}
and the unique solvability of the initial value problems in \eqref{3.8}. We summarize this discussion in the following result:

\begin{theorem} \lb{t3.1}
Assume Hypothesis \ref{h2.1} and let $\alpha_0 = (I_m \; 0)$. Suppose that the corresponding Weyl--Titchmarsh solutions of $D$ are denoted by 
$U_\pm(\zeta,\,\cdot\,,x_0,\alpha_0) 
= (u_{\pm,1}(\zeta,\,\cdot\,,x_0,\alpha_0) \;  
u_{\pm,2}(\zeta,\,\cdot\,,x_0,\alpha_0))^{\top}$ and the corresponding 
$m \times m$ matrix-valued half-line Weyl--Titchmarsh function of $D$ is given by $M^D_{\pm} (\,\cdot\,,x_0,\alpha_0)$. Then the $m \times m$ matrix-valued Weyl--Titchmarsh solutions of $H_j$, denoted by 
$\psi_{\pm,j} (z,\,\cdot\,,x_0)$, $j=1,2$, are given by \eqref{3.21} and 
\eqref{3.22}, and the $m \times m$ matrix-valued generalized Dirichlet-type 
Weyl--Titchmarsh functions $\hatt M_{\pm,0,j} (\,\, \cdot \, ,x_0)$ of $H_j$, $j=1,2$, are given by \eqref{3.23} and \eqref{3.24}. In particular,
\begin{equation}
\hatt M_{\pm,0,1} (z,x_0) = \zeta M^D_{\pm}(\zeta,x_0,\alpha_0) 
= - z \hatt M_{\pm,0,2} (z,x_0)^{-1}, \quad z = \zeta^2, \; 
\zeta \in \bbC \backslash \bbR.    \lb{3.28}
\end{equation}
\end{theorem}

A version of the equality 
$\hatt M_{\pm,0,1} (z,x_0) = - z \hatt M_{\pm,0,2} (z,x_0)^{-1}$, in the special scalar case $m=1$, and under the  stronger hypothesis 
$\phi \in AC_{\loc}(\bbR)$, first appeared in \cite[eq.\ (5.71)]{GNP97}, and was quoted again in \cite[eq.\ (A.25)]{GST96}.

The subscript ``$0$'' in $\hatt M_{\pm,0,j} (z,x_0)$, $j=1,2$, indicates that 
these generalized Weyl--Titchmarsh matrices correspond to a Dirichlet boundary condition at the reference point $x_0$ in the corresponding 
generalized half-line Schr\"odinger operators $H_{\pm,0,j}$, $j=1,2$, in 
$L^2([x_0,\pm\infty))^m$ defined by
\begin{align}
& (H_{\pm,0,j} u)(x) \no\\
&\quad = (\tau_j u)(x) 
= - \big(u^{[1,j]}\big)'(x) + (-1)^{j+1}\phi(x) u^{[1,j]} (x) 
  \text{ for a.e.\ $x \in [x_0,\pm\infty)$,}    \no \\
& \, u \in \dom(H_{\pm,0,j})  \no \\
& \quad = \big\{v \in L^2([x_0,\pm\infty))^m \, \big| \, v,  v^{[1,j]} \in AC([x_0,x_0\pm R])^m \text{ for all $R>0$};     \lb{3.29} \\
& \hspace*{2.6cm}  v(x_0) = 0; \, \big[\big(v^{[1,j]}\big)' + (-1)^j \phi v^{[1,j]}\big] \in 
L^2([x_0,\pm\infty))^m\big\},    \no \\
&\hspace*{9.75cm} j=1,2.\no
\end{align}

\noindent 
The corresponding Green's function of $H_{\pm,0,j}$ is then of the familiar form
\begin{align}
G_{+,0,j} (z,x,x') &= (H_{+,0,j} - z I)^{-1} (x,x')   \no \\
& = \begin{cases} s_j(z,x,x_0) \psi_{+,j} (\ol z,x',x_0)^*, & x \leq x', \\
\psi_{+,j} (z,x,x_0) s_j(\ol z,x',x_0)^*, & x' \leq x, \end{cases}    \lb{3.31} \\
& \;\;\, x,x' \in [x_0,\infty), \; z \in \bbC\backslash [0,\infty), \; j=1,2,     \no 
\end{align}
and
\begin{align}
G_{-,0,j} (z,x,x') &= (H_{-,0,j} - z I)^{-1} (x,x')   \no \\
& = - \begin{cases} s_j(z,x,x_0) \psi_{-,j} (\ol z,x',x_0)^*, & x' \leq x, \\
\psi_{-,j} (z,x,x_0) s_j(\ol z,x',x_0)^*, & x \leq x', \end{cases}    \lb{3.32} \\
& \;\;\;\, x,x' \in (-\infty,x_0], \; z \in \bbC\backslash [0,\infty), \; j=1,2.     \no 
\end{align} 
Similarly, the diagonal terms in \eqref{A.25} together with 
\eqref{2.17}, \eqref{2.26} and 
\eqref{3.21}--\eqref{3.25} yield the Green's function for $H_j$, 
\begin{align}
& G_j (z,x,x') = (H_j - z I)^{-1} (x,x')   \no \\
& \quad = \psi_{\mp,j} (z,x,x_0) 
\Big[\hatt M_{-,0,j}(z,x_0) - \hatt M_{+,0,j}(z,x_0)\Big]^{-1} 
\psi_{\pm,j} (\ol z,x',x_0)^*,     \lb{3.33} \\
& \hspace*{4cm} x \lesseqgtr x', \; 
x,x' \in \bbR, \; z \in \bbC\backslash [0,\infty), \; j=1,2.  \no 
\end{align}
One can show that 
\begin{align} 
\hatt M_{-,0,j}(z,x_0) - \hatt M_{+,0,j}(z,x_0) 
&= W(\psi_{+,j} (\ol z,\,\cdot\,,x_0)^* , \psi_{-,j} (z,\,\cdot\,,x_0))\lb{3.34}   \\
&= - W(\psi_{-,j} (\ol z,\,\cdot\,,x_0)^* , \psi_{+,j} (z,\,\cdot\,,x_0)),   \quad j=1,2,\no
\end{align} 
where $W(\cdot,\cdot)$, $j=1,2$, denote the Wronskians of matrix-valued functions
$F, G \in AC_{\loc}(\bbR)^{m \times m}$ defined by 
\begin{align}
W(F,G)(x) &= F(x) G'(x) - F'(x) G(x)   \lb{3.36dd}  \\
&= F(x) G^{[1,j]}(x) - F^{[1,j]}(x) G(x) +(-1)^j [F(x),\phi(x)]G(x),   \no \\ 
&\hspace*{5cm}\text{ for a.e.\ $x \in \bbR$,}\quad j=1,2.    \no
\end{align}
Of course, \eqref{3.21}--\eqref{3.28}, \eqref{3.31}--\eqref{3.33}, extend as usual to all $z \in \rho(H_{\pm,0,j})$, respectively, all $z \in \rho(H_j)$, $j=1,2$.

We note in passing that 
\begin{align}
& \, W(v_j(\ol z, \cdot)^*, u_j(z,\cdot))(x) \, \text{ is $x$-independent whenever} 
\no \\
& - [u_j'(z,x) + (-1)^{j+1} \phi(x) u_j(z,x)]' 
+ (-1)^{j+1} \phi(x) [u_j'(z,x)   \no \\
& \quad + (-1)^{j+1} \phi(x) u_j(z,x)] = z u_j(z,x),   \\
& - [v_j'(\ol z,x)^* + (-1)^{j+1} v_j(\ol z,x)^* \phi(x)]' 
+ (-1)^{j+1} [v_j'(\ol z,x)^*   \no \\ 
& \quad + (-1)^{j+1} v_j(\ol z,x)^* \phi(x)] \phi(x) = z v_j(\ol z, x)^*,    \quad
j=1,2.  \no 
\end{align} 

This circle of ideas will be further explored in Section \ref{s4}. 

\begin{remark} \lb{r3.2} 
In the particular case where $\phi \in AC_{\loc}(\bbR)^{m \times m}$ 
and hence, 
\begin{equation} 
V_j = \big[\phi^2 + (-1)^{j} \phi'\big] \in L^1_{\loc}(\bbR)^{m \times m}, \quad j=1,2, 
\end{equation} 
the relation between the generalized  Weyl--Titchmarsh matrices $\hatt M_{\pm,0,j}(\,\cdot\,,x_0)$, 
defined in \eqref{3.23}, \eqref{3.24}, and the standard (Dirichlet-type) Weyl--Titchmarsh matrices 
$M_{\pm,0,j}(\,\cdot\,,x_0)$ is especially simple and reads
\begin{equation}
\hatt M_{\pm,0,j}(\,\cdot\,,x_0) = M_{\pm,0,j}(\,\cdot\,,x_0) 
+ (-1)^{j+1} \phi(x_0),  \quad j=1,2.  
\end{equation}
In particular, since $\phi(x_0)$ is $z$-independent and self-adjoint, the function theoretic (and hence spectral theoretic) content of 
$\hatt M_{\pm,0,j}(\,\cdot\,,x_0)$ and $M_{\pm,0,j}(\,\cdot\,,x_0)$ coincides in this special case as they possess identical matrix measures in their respective Nevanlinna--Herglotz representations. (These matrix measures being generated by the half-line $m \times m$ matrix-valued spectral functions of $H_{\pm,0,j}$.)
\end{remark}

\begin{remark} \lb{r3.3}
In the particular scalar case $m=1$, and under the stronger assumption 
$\phi \in L^2_{\loc} (\bbR)$, the operators $H_j$, $j=1,2$, and especially, the associated Miura transformation, 
\begin{equation}
\phi \mapsto \phi^2 - \phi'
\end{equation}
(i.e., the relation between $\phi$ and $V_1$), was studied in great detail in \cite{KPST05}. However, the authors did not directly rely on $D$ and its   supersymmetric structure, but instead based their investigations on an oscillation theoretic approach using Hartman's notion of (non)principal solutions. Subsequently, the authors of \cite{FHMP09}, \cite{HMP11}, and 
\cite{HMP11a} used a Zakharov--Shabat (ZS), or Ablowitz--Kaup--Newell--Segur (AKNS) Dirac-type expression $\wti D$ in connection with their investigation of Miura transformations and inverse scattering theory. Explicitly, $\wti D$ represents a self-adjoint $L^2(\bbR)^2$-realization associated with the differential expression $\wti{\cD}$ of the form
\begin{equation}
\wti{\cD} = i \begin{pmatrix} - (d/dx) & - \phi(x) \\ \phi(x) & (d/dx) 
\end{pmatrix} \, \text{ for a.e.\ $x \in \bbR$.} 
\end{equation}
To make the connection with the supersymmetric formalism presented in this paper, we introduce the unitary $2 \times 2$ matrices
\begin{equation}
\Upsilon = \f{1}{2} \begin{pmatrix} 1-i & 1-i \\ 1+i & -1-i \end{pmatrix}, \quad 
\Upsilon^* = \f{1}{2} \begin{pmatrix} 1+i & 1-i \\ 1+i & -1+i \end{pmatrix} = \Upsilon^{-1},
\end{equation}
and observe that
\begin{equation}
\Upsilon \wti{\cD} \Upsilon^{-1} = \cD, 
\end{equation}
with $\cD$ in \eqref{2.6} the differential expression underlying the supersymmetric Dirac-type operator $D$.

It is the use of supersymmetry of $D$ in connection with the (standard) assumption of local integrability of the coefficient $\phi$ that instantly leads to Weyl--Titchmarsh solutions $U_{\pm}(\zeta,x,x_0,\alpha)$ and 
Weyl--Titchmarsh matrices $M^D_{\pm}(\zeta,x_0,\alpha)$ of $D$ and hence effortlessly via \eqref{3.21}--\eqref{3.25} to those of the generalized 
Schr\"odinger-type operators $H_{\pm,0,j}$, $H_j$, $j=1,2$. In particular, it immediately leads to the Green's functions \eqref{3.31}--\eqref{3.33} of $H_{\pm,0,j}$ and of $H_j$, $j=1,2$, respectively, and permits the more general hypothesis $\phi \in L^1_{\loc} (\bbR)$ rather than 
$\phi \in L^2_{\loc} (\bbR)$. In addition, it permits an effortless discussion of the matrix-valued case (the latter cannot easily be obtained via oscillation theoretic methods, cf.\ the comments preceding \cite[Hypothesis\ 3.6]{CG03} in this context). 
\end{remark}

\begin{remark} \lb{r3.4}
The supersymmetric formalism employed in this section relies on nonnegativity of 
$H_1 = A^* A$ and $H_2 = A A^*$ and so imposes a restriction on the distributional 
potential coefficients $V_j$ (formally, of the Miura-type 
$V_j(x) = \phi(x)^2 + (-1)^{j} \phi'(x)$), $j=1,2$. In practice, however, this restriction 
amounts to dealing with Schr\"odinger operators bounded from below as adding a sufficiently large positive constant to a given potential (if necessary) will render $H_1$ and 
$H_2$ nonnegative. (Subsequently, this additional constant can be removed.) 

In this context we also notice that it is possible to generalize the operator 
$D$ in \eqref{2.5} to the form
\begin{equation}
D_m = \begin{pmatrix} m & A^* \\ A & -m \end{pmatrix}, \quad m \in \bbR,    \lb{3.39}
\end{equation}
to the effect that then 
\begin{equation} 
D_m^2 = \begin{pmatrix} A^* A + m^2 I& 0 \\ 0 & A A^* + m^2 I \end{pmatrix} = 
(H_1 + m^2 I) \oplus (H_2 + m^2 I).    \lb{3.40}
\end{equation} 
\end{remark}

Finally, we note that these results on generalized Schr\"odinger operators with distributional potentials extend to general (three-coefficient) Sturm--Liouville operators, but we refrain from further details at this point.

\section{Basic Spectral Theory for $H_j$, $j=1,2$, and Some Applications}  \lb{s4}

In our final section we provide some spectral theoretic applications of the supersymmetric approach outlined in Section \ref{s3}. In particular, upon deriving the basic aspects of Weyl--Titchmarsh theory for the generalized Schr\"odinger operators $H_j$, $j=1,2$, in $L^2(\bbR)^m$ (cf.\ \eqref{3.5}, \eqref{3.6}) and a discussion of the corresponding spectral representations, we derive a local Borg--Marchenko uniqueness theorem by utilizing the known analog for the Dirac operator $D$. 

Since the Schr\"odinger operator $H_2$ as defined in \eqref{3.6} with coefficient  $\phi(\cdot)$ is realized as a Schr\"odinger operator of the form $H_1$ as defined in \eqref{3.5} with coefficient $\phi(\cdot)$ replaced by $-\phi(\cdot)$, it suffices to exclusively study spectral theory for Schr\"odinger operators of the form $H_1$.  In fact, the forms of many of the subsequent formulas relevant to the spectral theory of $H_j$, $j=1,2$, are independent of the choice $j\in \{1,2\}$, the exception being \eqref{4.67}, \eqref{4.81} below.

To set the stage for the main results of this section, we begin with a discussion of solutions to the equation $\tau_ju=zu$, $z\in \bbC\backslash\bbR$, $j=1,2$, with $\tau_j$ as defined in \eqref{3.15}, and their corresponding Wronskian relations.  

To this end, suppose $z_k\in \bbC\backslash \bbR$, $k=1,2$, and that $u_j(z_1,\cdot)$ and $u_j(z_2,\cdot)$ satisfy
\begin{equation}
u_j(z_k,\cdot),\, u^{[1,j]}_j(z_k,\cdot) \in AC_{\loc}(\bbR)^{m \times m},  \quad   
\tau_ju_j(z_k,\cdot)=z_ku_j(z_k,\cdot), \quad j, k=1,2.  \lb{4.1}
\end{equation}
Then one observes that
\begin{align}
&W\big(u_j(z_1,\cdot)^*,u_j(z_2,\cdot)\big)(x)=u_j(z_1,x)^*u^{[1,j]}_j(z_2,x)-\big(u^{[1,j]}_j(z_1,x)\big)^*u_j(z_2,x),\no\\
&\hspace*{10cm} j=1,2,\lb{4.2}
\end{align}
so that the Wronskian appearing in \eqref{4.2} is differentiable almost everywhere.  Moreover, \eqref{4.1} implies 
\begin{align}
\begin{split} 
\frac{d}{dx}\big(u^{[1,j]}_j(z_k,x)\big)=(-1)^{j+1}\phi(x)u^{[1,j]}_j(z_k,x)-z_ku_j(z_k,x) &\lb{4.3} \\
  \text{ for a.e.\ $x\in \bbR$}, \; j, k=1,2.& 
\end{split}   
\end{align}
As a result, one computes
\begin{align}
&\frac{d}{dx}W\big(u_j(z_1,\cdot)^*,u_j(z_2,\cdot)\big)(x) = 
\frac{d}{dx}\Big[u_j(z_1,x)^*u^{[1,j]}_j(z_2,x)-\big(u^{[1,j]}_j(z_1,x)\big)^*u_j(z_2,x) \Big]   \no \\
& \quad = (\overline{z_1}-z_2)u_j(z_1,x)^*u_j(z_2,x)  \, \text{ for a.e.\ $x\in \bbR$}.    \lb{4.4} 
\end{align}
We summarize the above considerations as follows: 

\begin{lemma}\lb{l4.1}
Assume Hypothesis \ref{h2.1} and suppose that $z_k\in \bbC\backslash \bbR$, $k=1,2$.  If $u_j(z_k,\cdot)$, $j, k=1,2$, satisfy \eqref{4.1}, then
\begin{equation}\lb{4.6}
\begin{split}
\frac{d}{dx}W\big(u_j(z_1,\cdot)^*,u_j(z_2,\cdot)\big)(x)=(\overline{z_1}-z_2)u_j(z_1,x)^*u_j(z_2,x)&\\
\text{for a.e.\ $x\in \bbR$}, \; j=1,2.&
\end{split}
\end{equation}
\end{lemma}

As an immediate consequence of Lemma \ref{l4.1}, one has the following result. 

\begin{corollary}\lb{c4.2}
Assume Hypothesis \ref{h2.1}. Then for $z\in\bbC\backslash\bbR$, the following identity holds.
\begin{align}
&\Im(z)\int_{x_0}^x dx'\,\psi_{\pm,j}(z,x',x_0)^*\psi_{\pm,j}(z,x',x_0)\no\\
&\quad =\Im\big(\hatt M_{\pm,0,j}(z,x_0) \big)-(2i)^{-1}
W\big(\psi_{\pm,j}(z,\, \cdot \, ,x_0)^*,\psi_{\pm,j}(z,\, \cdot \, ,x_0)\big)(x),\lb{4.8}\\
&\hspace*{8.2cm} x\in \bbR,\; j=1,2.\no
\end{align}
\end{corollary}
\begin{proof}
Let $z\in \bbC\backslash \bbR$ be fixed and choose $u_j(z_1,\cdot)=u_j(z_2,\cdot)
=\psi_{\pm,j}(z,\, \cdot \,,x_0)$, $j=1,2$, in \eqref{4.6}, one obtains
\begin{align}
&\frac{d}{dx'}W\big(\psi_{\pm,j}(z,\, \cdot \,,x_0)^*,\psi_{\pm,j}(z,\, \cdot \,,x_0)\big)(x')\no\\
&\quad=-2i\Im(z)\psi_{\pm,j}(z,x',x_0)^*\psi_{\pm,j}(z,x',x_0) \, \text{ for a.e.\ $x'\in \bbR$},\; j=1,2.\lb{4.10}
\end{align}
Integration of both sides of \eqref{4.10} from $x_0$ to $x$, using the normalizations in \eqref{3.26} and \eqref{3.27}, yields \eqref{4.8}. 
\end{proof}

\begin{lemma}\lb{l4.3}
Assume Hypothesis \ref{h2.1}. Then for any $z\in \bbC\backslash \bbR$, 
\begin{align}
\lim_{x\rightarrow \pm \infty}W\big(\psi_{\pm,j}(z,\, \cdot \,,x_0)^*, \psi_{\pm,j}(z,\cdot,x_0) \big)(x)&= 0,\quad j=1,2.\lb{4.11}
\end{align}
\end{lemma}
\begin{proof}
We provide a proof of \eqref{4.11} for the case $j=1$; an analogous argument is used to settle the case $j=2$.  In order to prove \eqref{4.11} for $j=1$, fix $z\in \bbC\backslash \bbR$ and observe that by \eqref{3.13} and \eqref{3.21},
\begin{align}
&W_1\big(\psi_{\pm,1}(z,\, \cdot \,,x_0)^*, \psi_{\pm,1}(z,\, \cdot \,,x_0) \big)(x)\no\\
&\quad=\psi_{\pm,1}(z,x,x_0)^*\psi_{\pm,1}^{[1,1]}(z,x,x_0)-\psi_{\pm,1}^{[1,1]}(z,x,x_0)^*\psi_{\pm,1}(z,x,x_0)\no\\
&\quad=\zeta u_{\pm,1}(\zeta,x,x_0,\alpha_0)^*u_{\pm,2}(\zeta,x,x_0,\alpha_0)-\zeta u_{\pm,2}(\zeta,x,x_0,\alpha_0)^*u_{\pm,1}(\zeta,x,x_0,\alpha_0)\no\\
&\quad=(\zeta-\overline{\zeta})u_{\pm,2}(\zeta,x,x_0,\alpha_0)^*u_{\pm,1}(\zeta,x,x_0,\alpha_0)\no\\
&\hspace*{.8cm}-\overline{\zeta}U_{\pm}(\zeta,x,x_0,\alpha_0)^*JU_{\pm}(\zeta,x,x_0,\alpha_0) 
\, \text{ for a.e.\ $x\in\bbR$}, \; z=\zeta^2.   \lb{4.13}
\end{align}
Since $\cD$ is in the limit point case one has the following limit relation (cf., e.g., \cite[Corollary 2.3]{HS84})
\begin{equation}\lb{4.14}
\lim_{x\rightarrow \pm\infty}U_{\pm}(\zeta,x,x_0,\alpha_0)^*JU_{\pm}(\zeta,x,x_0,\alpha_0)=0.
\end{equation}
Thus, in order to prove \eqref{4.11}, it suffices to show
\begin{equation}\lb{4.15}
\lim_{x\rightarrow \pm \infty}u_{\pm,2}(\zeta,x,x_0,\alpha_0)^*u_{\pm,1}(\zeta,x,x_0,\alpha_0)=0.
\end{equation}
To this end, one observes that the function under the limit in \eqref{4.15} is differentiable and that, in fact,
\begin{align}
&\big[u_{\pm,2}(\zeta,\, \cdot \, ,x_0,\alpha_0)^*u_{\pm,1}(\zeta,\, \cdot \,,x_0,\alpha_0)\big]'(x)\lb{4.16}\\
&\quad=\zeta u_{\pm,2}(\zeta,x,x_0,\alpha_0)^*u_{\pm,2}(\zeta,x,x_0,\alpha_0)-\overline{\zeta} u_{\pm,1}(\zeta,x,x_0,\alpha_0)^*u_{\pm,1}(\zeta,x,x_0,\alpha_0)    \no\\
&\hspace*{9.9cm}\text{for a.e.\ $x\in \bbR$.}    \no
\end{align}
We recall that $u_{\pm,2}(\zeta,\, \cdot \, ,x_0,\alpha_0),\, u_{\pm,1}(\zeta,\cdot,x_0,\alpha_0)\in L^2((0,\pm \infty))^{m\times m}$.\ As a result, one infers that 
\begin{equation}\lb{4.17}
u_{\pm,2}(\zeta,\, \cdot \, ,x_0,\alpha_0)^*u_{\pm,1}(\zeta,\cdot,x_0,\alpha_0)\in L^1((0,\pm\infty))^{m\times m},
\end{equation}
the same containment is true for the derivative by \eqref{4.16}.  Moreover,
\begin{align}
\begin{split} 
&u_{\pm,2}(\zeta, x,x_0,\alpha_0)^*u_{\pm,1}(\zeta,x,x_0,\alpha_0) 
=u_{\pm,2}(\zeta,0,x_0,\alpha_0)^*u_{\pm,1}(\zeta,0,x_0,\alpha_0)    \\
& \quad +\int_0^x dx'\, [u_{\pm,2}(\zeta,x',x_0,\alpha_0)^*u_{\pm,1}(\zeta,x',x_0,\alpha_0)]'(x'), 
\quad x\in \bbR,\lb{4.18}
\end{split} 
\end{align}
coupled with the fact that the function appearing under the integral in \eqref{4.18} belongs to $L^1((0,\pm\infty))$, affirms the existence of the limits appearing in \eqref{4.15}.  In light of \eqref{4.17}, both limits must equal zero.
\end{proof}

Taking limits $x\rightarrow \pm \infty$ throughout \eqref{4.8} and using \eqref{4.11}, one obtains the fundamental identities:
\begin{equation}\lb{4.19}
\begin{split}
\Im\big(\hatt M_{\pm,0,j}(z,x_0) \big)=\Im(z)\int_{x_0}^{\pm \infty}dx'\, \psi_{\pm,j}(z,x',x_0)^*\psi_{\pm,j}(z,x',x_0),&\\
z\in \bbC\backslash \bbR, \; j=1,2.&
\end{split}
\end{equation}
The identities in \eqref{4.19} show that $\pm\hatt M_{\pm,0,j}(z,x_0)$, $j=1,2$, are Nevanlinna--Herglotz functions.  We summarize this together with some other relevant properties of the generalized Dirichlet-type $m\times m$ matrix-valued Weyl--Titchmarsh functions $\hatt M_{\pm,0,j}(z,x_0)$, $j=1,2$, (associated to $H_j$) in the following result.

\begin{lemma}\lb{l4.4}
Assume Hypothesis \ref{h2.1} and let $\hatt M_{\pm,0,j}(z,x_0)$, $j=1,2$, $z\in \bbC\backslash \bbR$, denote the generalized Dirichlet-type $m\times m$ matrix-valued Weyl--Titchmarsh functions associated to $H_j$ as defined by \eqref{3.23} and \eqref{3.24}.  Then $\pm \hatt M_{\pm,0,j}(\, \cdot \, ,x_0)$, $j=1,2$, is an $m\times m$ matrix-valued Nevanlinna--Herglotz function of maximal rank $m$.  In particular,
\begin{align}
&\Im\big(\pm \hatt M_{\pm,0,j}(z,x_0)\big)\geq0,\quad z\in \bbC_+,\lb{4.21}\\
&\hatt M_{\pm,0,j}(\overline{z},x_0)=\hatt M_{\pm,0,j}(z,x_0)^*,\quad z\in \bbC\backslash \bbR, \lb{4.22}\\
&\rank\big(\hatt M_{\pm,0,j}(z,x_0)\big)=m,\quad z\in \bbC\backslash \bbR, \lb{4.23}\\
&\lim_{\e\downarrow 0}\hatt M_{\pm,0,j}(\lambda+i\e,x_0) \, \text{ exists for a.e.\ $\lambda \in \bbR$,} \; j=1,2.
\lb{4.24}
\end{align}
\end{lemma}
\begin{proof}
The inequalities in \eqref{4.21} follow immediately from \eqref{4.19}.  By \eqref{3.23} and \eqref{3.24}, $\hatt M_{\pm,0,j}(\, \cdot \,,x_0)$, $j=1,2$, are analytic on $\bbC\backslash \bbR$, thus they are Nevanlinna--Herglotz by \eqref{4.21}.  Relation \eqref{4.22} (resp., \eqref{4.23}) follows from \eqref{3.23} and \eqref{3.24} via \eqref{2.17} (resp., \eqref{2.18}).  Finally, \eqref{4.24} follows from the fact that $\pm \hatt M_{\pm,0,j}(\, \cdot \,,x_0)$, $j=1,2$, are Nevanlinna--Herglotz functions (cf., e.g., \cite{GT00}).  Alternatively, \eqref{4.24} can be immediately inferred from \eqref{3.23} and \eqref{3.24} together with \eqref{2.19}.
\end{proof}

\begin{remark}
Above, we used Corollary \ref{c4.2} and Lemma \ref{l4.3} to prove $\pm\hatt M_{\pm,0,j}(\, \cdot \,,x_0)$, $j=1,2$, are Nevanlinna--Herglotz functions.  Alternatively, one can use the following approach based on computing the imaginary parts for $\pm\hatt M_{\pm,0,j}(\, \cdot \,,x_0)$, $j=1,2$, directly using the known representations for $\Im\big(M_{\pm}^D(\zeta,c_0,\alpha_0) \big)$ (cf., \eqref{2.25}). We briefly sketch how this approach is carried out.

Note that the system
\begin{equation}\lb{5.10.1}
\cD\Psi=\zeta \Psi, \quad \zeta \in \bbC\backslash \bbR,
\end{equation}
can be recast as
\begin{equation}\lb{5.10.2}
J\Psi'=[\zeta I_{2m}+B]\Psi, \quad \zeta\in \bbC\backslash\bbR,
\end{equation}
where $J$ is the $2m\times 2m$ matrix defined in \eqref{2.7}, and 
\begin{equation}\lb{5.10.3}
B=B(x)=\begin{pmatrix}
0 & \phi(x)\\
\phi(x) & 0
\end{pmatrix} \, \text{ for a.e.\ $x\in \bbR$.}
\end{equation}
One can then verify by direct computation that if $\Psi_j=\Psi_j(\zeta_j,\, \cdot\,)$, $j=1,2$, denote solutions of \eqref{5.10.2} with $\zeta_j\in \bbC\backslash \bbR$, $j=1,2$, then
\begin{equation}\lb{5.10.4}
\big(\Psi_1^*\mathfrak{S}_1\Psi_2\big)'=-(\zeta_2+\overline{\zeta_1})\Psi_1^*\mathfrak{S}_3\Psi_2,
\end{equation}
where
\begin{equation}\lb{5.10.5}
\mathfrak{S}_1=\begin{pmatrix}
0 & I_m\\
I_m & 0
\end{pmatrix}, \quad 
\mathfrak{S}_3=\begin{pmatrix}
I_m & 0\\
0 & -I_m
\end{pmatrix}.
\end{equation}

Let $z\in \bbC\backslash \bbR$ and fix $\zeta\in \bbC$ with $\zeta^2=z$ and $\Im(\zeta)>0$.  Upon decomposing $\zeta$ and $\pm\hatt M_{\pm,0,1}(\, \cdot \, ,x_0)$ into its real and imaginary parts and using \eqref{3.23}, one computes
\begin{equation}\lb{5.10.6}
\Im\big(\pm\hatt M_{\pm,0,1}(z,x_0) \big)=\Im(\zeta)\Re\big(M_{\pm}^D(\zeta,x_0,\alpha_0) \big)+\Re(\zeta)\Im\big(M_{\pm}^D(\zeta,x_0,\alpha_0) \big).
\end{equation}
Choosing $\Psi_1=\Psi_2=U_{\pm}(\zeta,\, \cdot\, ,x_0,\alpha_0)$ and $\zeta_1=\zeta_2=\zeta$ in \eqref{5.10.4} and integrating over $[x_0,\infty)$ yields
\begin{align}
\Re\big(M_{\pm}^D(\zeta,x_0,\alpha_0) \big)&=\tfrac{1}{2}\lim_{x' \rightarrow \pm\infty}U_{\pm}(\zeta,x',x_0,\alpha_0^*)\mathfrak{S}_1U_{\pm}(\zeta,x',x_0,\alpha_0)\lb{5.10.7}\\
&\quad +\Re(\zeta)\int_{x_0}^{\pm\infty}dx\big[ u_{\pm,1}(\zeta,x,x_0,\alpha_0)^*u_{\pm,1}(\zeta,x,x_0,\alpha_0)\no\\
&\hspace*{2.95cm}-u_{\pm,2}(\zeta,x,x_0,\alpha_0)^*u_{\pm,2}(\zeta,x,x_0,\alpha_0)\big]     \no \\
& \hspace*{-2.5cm} =\Re(\zeta)\int_{x_0}^{\pm\infty}dx\big[ u_{\pm,1}(\zeta,x,x_0,\alpha_0)^*u_{\pm,1}(\zeta,x,x_0,\alpha_0) 
\no \\
&\hspace*{-2.5cm} \quad -u_{\pm,2}(\zeta,x,x_0,\alpha_0)^*u_{\pm,2}(\zeta,x,x_0,\alpha_0)\big],  \lb{5.10.8}
\end{align}
applying \eqref{4.15}. The fact that the supersymmetric nature of $D$ permits the representation 
\eqref{5.10.8} for the real part of $M_{\pm}^D(\zeta,x_0,\alpha_0)$ appears to have gone unnoticed in the literature. By \eqref{2.25}, the representation in \eqref{5.10.6} can be recast as
\begin{align}
\Im\big(\hatt M_{\pm,0,1}(z,x_0)\big)&=2\Im(\zeta)\Re(\zeta)\int_{x_0}^{\pm \infty}dx\, u_{\pm,1}(\zeta,x,x_0,\alpha_0)^*u_{\pm,1}(\zeta,x,x_0,\alpha_0)\no\\
&= \Im(z)\int_{x_0}^{\pm \infty}dx\, \psi_{\pm,1}(z,x,x_0,\alpha_0)^*\psi_{\pm,1}(z,x,x_0,\alpha_0),\lb{5.10.9}
\end{align}
(making use of \eqref{3.21}), implying \eqref{4.19}. The result for $\hatt M_{\pm,0,2}(z,x_0)$ follows similarly. 
\end{remark}

In order to establish spectral theory for $H_j$, we introduce the $2m\times 2m$ matrix-valued Weyl--Titchmarsh matrix, $\hatt {\mathbf{M}}_j(z,x_0)\in \bbC^{2m\times 2m}$, $z\in \bbC\backslash \bbR$, associated to $H_j$, $j=1,2$, as follows
\begin{align}
\hatt {\mathbf{M}}_j(z,x_0)&=\Big(\hatt {\mathbf{M}}_{j,k,k'}(z,x_0)\Big)_{k,k'=0,1}, \quad z\in\bbC\backslash\bbR, \; j=1,2, \lb{4.25}
\\
\hatt {\mathbf{M}}_{j,0,0}(z,x_0) &= W(z)^{-1},\lb{4.26}
\\
\hatt {\mathbf{M}}_{j,0,1}(z,x_0) &= 2^{-1} W(z)^{-1} \Big[\hatt M_{-,0,j}(z,x_0)+\hatt M_{+,0,j}(z,x_0)\Big],\lb{4.27}
\\
\hatt {\mathbf{M}}_{j,1,0}(z,x_0) &= 2^{-1} \Big[\hatt M_{-,0,j}(z,x_0)+\hatt M_{+,0,j}(z,x_0)\Big] W(z)^{-1},\lb{4.28}
\\
\hatt {\mathbf{M}}_{j,1,1}(z,x_0) &= \hatt M_{\pm,0,j}(z,x_0) W(z)^{-1} \hatt M_{\mp,0,j}(z,x_0),\lb{4.29}
\end{align}
where we have used the abbreviation (cf.\ \eqref{3.36dd})
\begin{align}
\begin{split}
W(z)&=W(\psi_{+,j}(\overline{z},\, \cdot \,,x_0)^*,\psi_{-,j}(z,\, \cdot \,,x_0))  \\
&= \hatt M_{-,0,j}(z,x_0) - \hatt M_{+,0,j}(z,x_0),\quad z\in \bbC\backslash \bbR, \; j=1,2.
\end{split} 
\end{align}

With \eqref{3.28} and the definitions in \eqref{4.25}--\eqref{4.29}, one readily verifies that
\begin{align}
\hatt {\mathbf{M}}_1(z,x_0)&=\begin{pmatrix}\zeta^{-1}I_m & 0 \\ 0 & I_m \end{pmatrix}\mathbf{M}^D(\zeta,x_0,\alpha_0)\begin{pmatrix}I_m & 0 \\ 0 & \zeta I_m \end{pmatrix},\lb{4.67}\\
\hatt {\mathbf{M}}_2(z,x_0)&=\begin{pmatrix}-\zeta^{-1}M_+^D(\zeta,x_0,\alpha_0) & 0 \\ 0 & M_-^D(\zeta,x_0,\alpha_0)^{-1} \end{pmatrix}\mathbf{M}^D(\zeta,x_0,\alpha_0)\lb{4.81}\\
&\quad \times\begin{pmatrix}-M_-^D(\zeta,x_0,\alpha_0) & 0 \\ 0 & \zeta M_+^D(\zeta,x_0,\alpha_0)^{-1} \end{pmatrix}, \quad \zeta^2=z,\, \zeta\in \bbC\backslash \bbR.\no
\end{align}

In addition, one notes that $\hatt {\mathbf{M}}_j(z,x_0)$ is a $\bbC^{2m\times 2m}$-valued Nevanlinna--Herglotz matrix with representation
\begin{align}
\begin{split}
& \hatt {\mathbf M}_j(z,x_0)=\mathbf{C}_j(x_0)+\int_{\bbR}
d\hatt {\mathbf{\Omega}}_j (\lambda,x_0)\bigg[\frac{1}{\lambda -z}-\frac{\lambda}
{1+\lambda^2}\bigg], \quad z\in\bbC\backslash\bbR, \lb{4.30} \\
& \mathbf{F}_j(x_0)=\mathbf{F}_j(x_0)^*, \quad \int_{\bbR}
\big\|d\hatt {\mathbf{\Omega}}_j(\lambda,x_0)\big\|_{\bbC^{2m}}(1+\lambda^2)^{-1} <\infty, \quad j=1,2.
\end{split}
\end{align}
The Stieltjes inversion formula for the nonnegative $2m\times 2m$ matrix-valued measure $d\hatt {\mathbf{\Omega}}_j(\, \cdot \, ,x_0)$ then reads
\begin{align}
\hatt {\mathbf{\Omega}}_j((\lambda_1,\lambda_2],x_0)
=\frac{1}{\pi} \lim_{\delta\downarrow 0}
\lim_{\varepsilon\downarrow 0} \int^{\lambda_2+\delta}_{\lambda_1+\delta}
d\lambda \, \Im\big(\hatt {\mathbf{M}}_j(\lambda +i\varepsilon,x_0)\big),&\lb{4.31}\\
 \lambda_1, \lambda_2 \in\bbR, \; \lambda_1<\lambda_2,\, j=1,2.&\no
\end{align}
In particular, $d\hatt {\mathbf{\Omega}}_j(\, \cdot \, ,x_0)$, $j=1,2$, is a $2\times 2$ block matrix-valued measure with $\bbC^{m\times m}$-valued entries $d\hatt {\mathbf{\Omega}}_{j,\ell,\ell'}(\, \cdot \,,x_0)$, $\ell,\ell'=0,1$. Since the diagonal entries of $\hatt {\mathbf M}_j(\, \cdot \,,x_0)$ are Nevanlinna--Herglotz functions, the diagonal entries of the measure $d\hatt {\mathbf{\Omega}}_j(\, \cdot \,,x_0)$ are nonnegative $\bbC^{m\times m}$-valued measures. The off-diagonal entries of the measure $d\hatt {\mathbf{\Omega}}_j(\, \cdot \,,x_0)$ naturally admit decompositions into a linear combination of four nonnegative matrix-valued measures.

Next, we relate the family of spectral projections $\{E_{H_j}(\lambda)\}_{\lambda\in \bbR}$ of the self-adjoint operator $H_j$ and the $2m\times 2m$ matrix-valued increasing spectral function $\hatt {\mathbf{\Omega}}_j(\lambda, x_0)$, $\lambda \in \bbR$, which generates the matrix-valued measure in the Nevanlinna--Herglotz representation \eqref{4.30} of $\hatt {\mathbf{M}}_j(\, \cdot \,,x_0)$, $j=1,2$.  

We note that for $F\in C(\bbR)$,
\begin{align}
&\big(f,F(H_j)g\big)_{L^2(\bbR)^m}= \int_{\bbR}
d\,\big(f,E_{H_j}(\lambda)g\big)_{L^2(\bbR)^m}\,
F(\lambda), \lb{4.32} \\
& f, g \in\dom(F(H_j)) =\bigg\{h\in L^2(\bbR)^m \,\bigg|\,
\int_{\bbR} d \|E_{H_j}(\lambda)h\|_{L^2(\bbR)^m}^2 \, |F(\lambda)|^2
< \infty\bigg\}, \no\\
&\hspace*{10.1cm}j=1,2.\no
\end{align}

\begin{theorem} \lb{t4.5}
Assume Hypothesis \ref{h2.1} and let $f,g \in C^\infty_0(\bbR)^m$,
$F\in C(\bbR)$, $x_0\in\bbR$, and $\lambda_1, \lambda_2 \in\bbR$,
$\lambda_1<\lambda_2$. Then,
\begin{align} \lb{4.33}
&\big(f,F(H_j)E_{H_j}((\lambda_1,\lambda_2])g\big)_{L^2(\bbR)^m} 
\no \\
&\quad =
\big(\hatt
f_j(\, \cdot \,,x_0),M_FM_{\chi_{(\lambda_1,\lambda_2]}} \hatt
g_j(\, \cdot \,,x_0)\big)_{L^2(\bbR;d\hatt {\mathbf{\Omega}}_j(\, \cdot \,,x_0))},\quad j=1,2,
\end{align}
where we introduced the notation
\begin{align} \lb{4.34}
&\hatt h_{j,0}(\lambda,x_0) = \int_\bbR dx \,
c_j(\lambda,x,x_0)^* h(x),  \quad
\hatt h_{j,1}(\lambda,x_0) = \int_\bbR dx \,
s_j(\lambda,x,x_0)^* h(x)
\no
\\
&\hatt h_j(\lambda,x_0) = \big(\,\hatt h_{j,0}(\lambda,x_0),
\hatt h_{j,1}(\lambda,x_0)\big)^\top,  \quad
\lambda \in\bbR, \;  h\in C^\infty_0(\bbR)^m,\quad j=1,2,
\end{align}
and $M_G$ denotes the maximally defined operator of multiplication
by the function $G \in C(\bbR)$ in the
Hilbert space $L^2(\bbR;d\hatt {\mathbf{\Omega}}_j(\, \cdot \,,x_0))$,
\begin{align}
\begin{split}
& \big(M_G\hatt h\big)(\lambda)=G(\lambda)\hatt h(\lambda)
=\big(G(\lambda) \hatt h_0(\lambda), G(\lambda) \hatt h_1(\lambda)\big)^\top
\, \text{ for \ $\hatt {\mathbf{\Omega}}_j(\, \cdot \,,x_0)$-a.e.\ $\lambda\in\bbR$}, \lb{4.35} \\
& \hatt h\in\dom(M_G)=\big\{\hatt k \in
L^2(\bbR;d\hatt {\mathbf{\Omega}}_j(\, \cdot \, ,x_0)) \,\big|\,
G\hatt k \in L^2(\bbR;d\hatt {\mathbf{\Omega}}_j(\, \cdot \,,x_0))\big\},\quad j=1,2.
\end{split}
\end{align}
\end{theorem}
\begin{proof} 
We fix $j \in \{1,2\}$. Using the weak version of Stone's formula, one obtains
\begin{align}
&\big(f,F(H_j)E_{H_j}((\lambda_1,\lambda_2])g\big)_{L^2(\bbR)^m}  
\no \\
& \quad = \lim_{\delta\downarrow 0}\lim_{\varepsilon\downarrow 0}
\frac{1}{2\pi i} \int_{\lambda_1+\delta}^{\lambda_2+\delta}
d\lambda \, F(\lambda) \big[\big(f,(H_j-(\lambda+i\varepsilon)I_{L^2(\bbR)^m})^{-1}g\big)_{L^2(\bbR)^m}  
\no \\
& \hspace*{4.9cm} - \big(f,(H_j-(\lambda-i\varepsilon)I_{L^2(\bbR)^m})^{-1}g\big)_{L^2(\bbR)^m}\big].  
\lb{4.36} 
\end{align}
Using that the resolvent of $H_j$ is an integral operator with kernel \eqref{3.33} in \eqref{4.36}, and freely interchanging the $dx$ and $dx'$ integrals with the limits and the
$d\lambda$ integral (since all integration domains are finite and all
integrands are continuous), and employing the expressions \eqref{3.25} for
$\psi_{\pm,j}(z,x,x_0)$, one obtains
\begin{align}
&\big(f,F(H_j)E_{H_j}((\lambda_1,\lambda_2])g\big)_{L^2(\bbR)^m}
=\int_{\bbR} dx \bigg(f(x), \bigg\{ \int_{-\infty}^x dx' \, \no
\\
& \quad \times \lim_{\delta\downarrow
0}\lim_{\varepsilon\downarrow 0} \frac{1}{2\pi i}
\int_{\lambda_1+\delta}^{\lambda_2+\delta} d\lambda \, F(\lambda)
\Big[\big[c_j(\lambda,x,x_0) +
s_j(\lambda,x,x_0) \hatt M_{+,0,j}(\lambda+i\varepsilon,x_0) \big] \no
\\
& \hspace*{1.2cm} \times W(\lambda+i\varepsilon)^{-1}
\big[c_j(\lambda,x',x_0)^* + \hatt M_{-,0,j}(\lambda+i\varepsilon,x_0) s_j(\lambda,x',x_0)^* \big]
g(x') \no
\\
& \qquad -\big[c_j(\lambda,x,x_0) +
s_j(\lambda,x,x_0) \hatt M_{+,0,j}(\lambda-i\varepsilon,x_0) \big] \no
\\
& \hspace*{1.2cm} \times W(\lambda-i\varepsilon)^{-1}
\big[c_j(\lambda,x',x_0)^* + \hatt M_{-,0,j}(\lambda-i\varepsilon,x_0) s_j(\lambda,x',x_0)^* \big]
g(x') \Big] \no
\\
& \quad +\int_x^\infty dx'\, \lim_{\delta\downarrow 0}
\lim_{\varepsilon\downarrow 0} \frac{1}{2\pi i}
\int_{\lambda_1+\delta}^{\lambda_2+\delta} d\lambda \, F(\lambda)
\lb{4.38} \\
& \qquad \times \Big[\big[c_j(\lambda,x,x_0) +
s_j(\lambda,x,x_0) \hatt M_{-,0,j}(\lambda+i\varepsilon,x_0) \big]
\no \\
& \hspace*{1.2cm} \times W(\lambda+i\varepsilon)^{-1}
\big[c_j(\lambda,x',x_0)^* +
\hatt M_{+,0,j}(\lambda+i\varepsilon,x_0) s_j(\lambda,x',x_0)^* \big]
g(x') \no
\\
& \qquad -\big[c_j(\lambda,x,x_0) +
s_j(\lambda,x,x_0) \hatt M_{-,0,j}(\lambda-i\varepsilon,x_0) \big]
\no
\\
& \hspace*{1.2cm} \times W(\lambda-i\varepsilon)^{-1} \big[c_j(\lambda,x',x_0)^* \no\\
&\hspace*{1.7cm}+
\hatt M_{+,0,j}(\lambda-i\varepsilon,x_0) s_j(\lambda,x',x_0)^* \big]
g(x') \Big]\bigg\}\bigg)_{\bbC^m}.   \no
\end{align}
Here we employed \eqref{4.22}, the fact that for fixed $x\in\bbR$,
$c_j(z,x,x_0)$ and $s_j(z,x,x_0)$ are entire
with respect to $z$, that $c_j(z,\, \cdot \, ,x_0), s_j(z,\, \cdot \, ,x_0) \in AC_{\loc}(\bbR;\cH)$, and hence that
\begin{align}
\begin{split} 
c_j(\lambda\pm i\varepsilon,x,x_0)
&\underset{\varepsilon\downarrow 0}{=}
c_j(\lambda,x,x_0) \pm
i\varepsilon (d/dz)c_j(z,x,x_0)|_{z=\lambda} + \Oh(\varepsilon^2),   \\
s_j(\lambda\pm i\varepsilon,x,x_0)
&\underset{\varepsilon\downarrow 0}{=} s_j(\lambda,x,x_0)
\pm i\varepsilon (d/dz)s_j(z,x,x_0)|_{z=\lambda}
+ \Oh(\varepsilon^2),   \lb{4.39} 
\end{split}  
\end{align}
with $\Oh(\varepsilon^2)$ being uniform with respect to $(\lambda,x)$ as long as $\lambda$ and $x$ vary in compact subsets of $\bbR$. Moreover, we used that
\begin{align}
&\varepsilon\big\|\hatt {\mathbf{M}}_j(\lambda+i\varepsilon,x_0)\big\|_{\bbC^{2m\times 2m}}\leq
C(\lambda_1,\lambda_2,\varepsilon_0,x_0), \quad \lambda\in
[\lambda_1,\lambda_2], \; 0<\varepsilon\leq\varepsilon_0, \no \\
&\varepsilon
\big\|\Re\big(\hatt {\mathbf{M}}_j(\lambda+i\varepsilon,x_0)\big)\big\|_{\bbC^{2m\times 2m}}
\underset{\varepsilon\downarrow 0}{=}\oh(1), \quad \lambda\in\bbR,     \lb{4.40}
\end{align}
since $\hatt {\mathbf{M}}_j(\,\cdot \,,x_0)$, are $\bbC^{2m\times 2m}$-valued Nevanlinna--Herglotz functions.  Moreover, we utilized \eqref{4.22}, \eqref{4.39},
\eqref{4.40}, and the elementary facts
\begin{align} \lb{4.41}
&\Im\big[\hatt M_{\pm,0,j}(\lambda+i\varepsilon,x_0) W(\lambda+i\varepsilon)^{-1}\big]
\no
\\
&\quad = \f{1}{2}
\Im\big[[\hatt M_{-,0,j}(\lambda+i\varepsilon,x_0) + \hatt M_{+,0,j}(\lambda+i\varepsilon,x_0)] W(\lambda+i\varepsilon)^{-1}\big],
\no
\\
&\Im\big[W(\lambda+i\varepsilon)^{-1} \hatt M_{\pm,0,j}(\lambda+i\varepsilon,x_0)\big] \\
&\quad = \f{1}{2}
\Im\big[W(\lambda+i\varepsilon)^{-1} [\hatt M_{-,0,j}(\lambda+i\varepsilon,x_0) + \hatt M_{+,0,j}(\lambda+i\varepsilon,x_0)]\big],
\quad \lambda\in\bbR, \; \varepsilon >0.    \no
\end{align}
Collecting appropriate terms in \eqref{4.38} then yields
\begin{align} \lb{4.42}
&\big(f,F(H_j)E_{H}((\lambda_1,\lambda_2])g\big)_{L^2(\bbR)^m}
= \int_\bbR dx \bigg(f(x), \int_\bbR dx'
\lim_{\delta\downarrow 0}\lim_{\varepsilon\downarrow 0}
\frac{1}{\pi} \int_{\lambda_1+\delta}^{\lambda_2+\delta} d\lambda
\, F(\lambda)
\no \\
& \ \times\Big\{
c_j(\lambda,x,x_0)
\Im\big[W(\lambda+i\varepsilon)^{-1}\big]
c_j(\lambda,x',x_0)^*
\\
& \quad
+ 2^{-1} c_j(\lambda,x,x_0)
\no \\
& \qquad
\times \Im\big[W(\lambda+i\varepsilon)^{-1} [\hatt M_{-,0,j}(\lambda+i\varepsilon,x_0) + \hatt M_{+,0,j}(\lambda+i\varepsilon,x_0)]\big]
s_j(\lambda,x',x_0)^*
\no \\
& \quad
+ 2^{-1} s_j(\lambda,x,x_0)
\no \\
& \qquad
\times \Im\big[[\hatt M_{-,0,j}(\lambda+i\varepsilon,x_0)+
\hatt M_{+,0,j}(\lambda+i\varepsilon,x_0)] W(\lambda+i\varepsilon)^{-1}\big] c_j(\lambda,x',x_0)^*
\no \\
& \quad
+s_j(\lambda,x,x_0)
\no \\
& \qquad
\times 
\Im\big[\hatt M_{-,0,j}(\lambda+i\varepsilon,x_0)
W(\lambda+i\varepsilon)^{-1}\no\\
&\hspace*{1.7cm}\times
\hatt M_{+,0,j}(\lambda+i\varepsilon,x_0)\big]
s_j(\lambda,x',x_0)^*\Big\} g(x') \bigg)_{\bbC^m}.    \no
\end{align}
Since by \eqref{4.31} (for $\ell, \ell'=0,1$)
\begin{align}
\begin{split}
&\int_{(\lambda_1,\lambda_2]} d\hatt {\mathbf{\Omega}}_{j,\ell,\ell'}(\lambda,x_0)
= \hatt {\mathbf{\Omega}}_{j,\ell,\ell'}((\lambda_1,\lambda_2],x_0) \\
& \quad =
\lim_{\delta\downarrow 0}\lim_{\varepsilon\downarrow 0}
\frac{1}{\pi}\int_{\lambda_1+\delta}^{\lambda_2+\delta} d\lambda \,
\Im\big(\hatt {\mathbf{M}}_{j,\ell,\ell'}(\lambda+i\varepsilon,x_0)\big),     \lb{4.43}
\end{split}
\end{align}
one also has (again for $\ell, \ell' \in\{0,1\}$)
\begin{align}
&\int_{\bbR} d\hatt\Omega_{j,\ell,\ell'}(\lambda,x_0)\, h(\lambda) =
\lim_{\varepsilon\downarrow 0} \frac{1}{\pi}\int_{\bbR} d\lambda \,
\Im\big(\hatt {\mathbf{M}}_{j,\ell,\ell'}(\lambda+i\varepsilon,x_0)\big)\, h(\lambda),
\quad h\in C_0(\bbR)^m,  \lb{4.44} \\
&\int_{(\lambda_1,\lambda_2]} d\hatt {\mathbf{\Omega}}_{j,\ell,\ell'}(\lambda,x_0)
\, k(\lambda) =
\lim_{\delta\downarrow 0} \lim_{\varepsilon\downarrow 0} \frac{1}{\pi}
\int_{\lambda_1+\delta}^{\lambda_2+\delta} d\lambda \,
\Im\big(\hatt {\mathbf{M}}_{j,\ell,\ell'}(\lambda+i\varepsilon,x_0)\big)\, k(\lambda), \no
\\
& \hspace*{9.2cm} k\in C(\bbR)^m. \lb{4.45}
\end{align}
Then using \eqref{4.25}--\eqref{4.29}, \eqref{4.34}, and interchanging the $dx$, $dx'$ and $d\hatt {\mathbf{\Omega}}_{j,\ell,\ell'}(\, \cdot \,,x_0)$, $\ell,\ell'=0,1$, integrals
once more, one concludes from \eqref{4.42} that 
\begin{align}
\begin{split} 
&\big(f,F(H_j)E_{H_j}((\lambda_1,\lambda_2])g\big)_{L^2(\bbR)^m}    \\
& \quad = \int_{(\lambda_1,\lambda_2]} F(\lambda) \, \big(\hatt
f_j(\lambda,x_0), d\hatt {\mathbf{\Omega}}_j(\lambda,x_0) \,
\hatt g_j(\lambda,x_0) \big)_{\bbC^{2m}},     \lb{4.46}
\end{split} 
\end{align}
implying \eqref{4.33}.
\end{proof}

Next, we improve on Theorem \ref{t4.5} and remove the compact support restrictions on $f$ 
and $g$ in the usual way, closely following and appropriately adapting the argument of \cite[(2.46)--(2.67)]{GZ06}. This leads to a variant of the spectral theorem for (functions of) $H_j$, $j=1,2$.
We consider the map
\begin{align}
&\widetilde U_j(x_0) : \begin{cases} C_0^\infty(\bbR)^m\to
L^2(\bbR;d\hatt {\mathbf{\Omega}}_j(\, \cdot \,,x_0))
\\[1mm]
h \mapsto \hatt h_j(\, \cdot \,,x_0)
=\big(\,\hatt h_{j,0}(\lambda,x_0),
\hatt h_{j,1}(\lambda,x_0)\big)^\top, \end{cases} \lb{4.47}
\\
&  \hatt h_{j,0}(\lambda,x_0)=\int_\bbR dx \,
c_j(\lambda,x,x_0)^* h(x), \quad \hatt h_{j,1}(\lambda,x_0)=\int_\bbR dx \, s_j(\lambda,x,x_0)^* h(x).   \no
\end{align}
Taking $f=g$, $F=1$, $\lambda_1\downarrow -\infty$, and
$\lambda_2\uparrow \infty$ in \eqref{4.33} then shows that $\widetilde
U_j(x_0)$, $j=1,2$, are densely defined isometries in $L^2(\bbR)^m$,
which extend by continuity to isometries on $L^2(\bbR)^m$. The latter are denoted by $U_j(x_0)$ and are defined by
\begin{align}
&U_j(x_0) : \begin{cases} L^2 (\bbR)^m\to
L^2(\bbR;d\hatt {\mathbf{\Omega}}_j(\, \cdot \,,x_0)) \\[1mm]
h \mapsto \hatt h_j(\, \cdot \,,x_0)
= \big(\,\hatt h_{j,0}(\, \cdot \,,x_0),
\hatt h_{j,1}(\, \cdot \,,x_0)\big)^\top, \end{cases} \lb{4.48} \\
& \hatt h_j(\, \cdot \,,x_0)=\begin{pmatrix}
\hatt h_{j,0}(\, \cdot \,,x_0) \\
\hatt h_{j,1}(\, \cdot \,,x_0) \end{pmatrix}=
\slimes_{a\downarrow -\infty, b \uparrow\infty} \begin{pmatrix}
\int_{a}^b dx \, c_j(\, \cdot \,,x,x_0)^* h(x) \\
\int_{a}^b dx \, s_j(\, \cdot \,,x,x_0)^* h(x) \end{pmatrix}, \no 
\end{align}
where $\slimes$ refers to the $L^2(\bbR;d\hatt {\mathbf{\Omega}}_j(\, \cdot \, ,x_0))$-limit.

\begin{theorem} \lb{t4.6} 
Assume Hypothesis \ref{h2.1} and let $F\in C(\bbR)$ and $x_0\in\bbR$. Then,
\begin{equation}\lb{4.63}
U_j(x_0) F(H_j)U_j(x_0)^{-1} = M_F,\quad j=1,2,
\end{equation}
in $L^2(\bbR;d\hatt {\mathbf{\Omega}}_j(\, \cdot \,,x_0))$ $($cf.\ \eqref{4.35}$)$.
Moreover,
\begin{equation}\lb{4.64}
\sigma(H_j)=\supp \big(d\hatt {\mathbf{\Omega}}_j(\, \cdot \,,x_0)\big),\quad j=1,2,
\end{equation}
and the multiplicity of the spectrum of $H_j$, $j=1,2$, is at most equal to $2m$.
\end{theorem}
\begin{proof}
Again, we fix $j \in \{1,2\}$. One observes that the calculation in \eqref{4.46} yields
\begin{align}
\begin{split} 
(E_{H_j}((\lambda_1,\lambda_2])g)(x) = \int_{(\lambda_1,\lambda_2]}
(c_j(\lambda,x,x_0), s_j(\lambda,x,x_0)) \,
d\hatt {\mathbf{\Omega}}_j(\lambda,x_0)\,
\hatt g_j(\lambda,x_0),&  \\
g\in C_0^\infty(\bbR)^m,& \lb{4.49}
\end{split} 
\end{align}
and, as a result, extends to all $g\in L^2(\bbR)^m$
by continuity. Moreover, taking $\lambda_1\downarrow -\infty$ and
$\lambda_2\uparrow \infty$ in \eqref{4.49} and using the spectral family properties
$\slim_{\lambda\downarrow -\infty} E_{H_j}(\lambda)=0$, 
$\slim_{\lambda\uparrow \infty} E_{H_j}(\lambda)=I_{L^2(\bbR)^m}$, where
$E_{H_j}(\lambda)=E_{H_j}((-\infty,\lambda])$, $\lambda\in\bbR$, then yields
\begin{align}
\begin{split} 
g(\cdot) =\slimes_{\mu_1\downarrow -\infty, \mu_2\uparrow \infty}
\int_{(\mu_1,\mu_2]} (c_j(\lambda,\, \cdot \, ,x_0),  
s_j(\lambda,\, \cdot \, ,x_0)) \, d\hatt {\mathbf{\Omega}}_j(\lambda,x_0),&   \\
\hatt g_j(\lambda,x_0), \quad g\in L^2(\bbR)^m,  & \lb{4.52}
\end{split} 
\end{align}
where $\slimes$ here refers to the limit in $L^2(\bbR)^m$.  Next, we show
that the maps $U_j(x_0)$ in \eqref{4.48} are onto and hence that
$U_j(x_0)$ are unitary maps with
\begin{align}
&U_j(x_0)^{-1} : \begin{cases}
L^2(\bbR;d\hatt {\mathbf{\Omega}}_j(\, \cdot \,,x_0)) \to  L^2(\bbR)^m  \\[1mm]
\hatt h \mapsto  h_j, \end{cases} \lb{4.53} \\
& h_j(\cdot)= \slimes_{\mu_1\downarrow -\infty, \mu_2\uparrow
\infty} \int_{(\mu_1,\mu_2]}(c_j(\lambda,\, \cdot \,,x_0),
s_j(\lambda,\, \cdot \,,x_0)) \, d\hatt {\mathbf{\Omega}}_j(\lambda,x_0)\,
\hatt h(\lambda).    \no
\end{align}
Letting $W_j(x_0)$ temporarily denote the operators defined by \eqref{4.53}, one infers that $W_j(x_0)$ are bounded.  Indeed, for any $\hatt f\in C_0^{\infty}(\bbR)^{2m}$, $g\in C_0^{\infty}(\bbR)^{m}$, one computes
\begin{align}
&\big(g,W_j(x_0)\hatt f\,\big)_{L^2(\bbR)^m}\no\\
&\quad=\int_{\bbR}dx\, \bigg(g(x),\int_{\bbR}\big(c_j(\lambda,x, x_0),s_j(\lambda,x,x_0)\big)d\hatt {\mathbf{\Omega}}_j(\lambda,x_0)\hatt f(\lambda) \bigg)_{\bbC^m}\no\\
&\quad=\int_{\bbR}\int_{\bbR}dx\,\bigg(\big(c_j(\lambda,x, x_0),s_j(\lambda,x,x_0)\big)^*g(x),d\hatt {\mathbf{\Omega}}_j(\lambda,x_0)\hatt f(\lambda) \bigg)_{\bbC^{2m}}\no\\
&\quad=(U_j(x_0)g,\hatt f)_{L^2(\bbR;d\hatt {\mathbf{\Omega}}_j(\, \cdot \, ,x_0))}.\lb{4.54}
\end{align}
Since $U_j(x_0)$ are isometries, \eqref{4.54} extends by continuity to all $g\in L^2(\bbR)^{m}$.  Thus,
\begin{align}
\big\|W_j(x_0) \hatt f \,\big\|_{L^2(\bbR)^{m}}
&= \sup_{g\in  L^2(\bbR)^{m}, \, g\neq 0}
\bigg|\f{\big(g,W_j(x_0) \hatt f \,\big)_{L^2(\bbR)^{m}}}{\|g\|_{L^2(\bbR)^{m}}}
\bigg| \no \\
& \leq  \sup_{g\in  L^2(\bbR)^{m}, \, g\neq 0}
\f{\|U_j(x_0) g\|_{L^2(\bbR;d\hatt {\mathbf{\Omega}}_j(\, \cdot \,,x_0))}}
{\|g\|_{L^2(\bbR)^m}}
\big\|\hatt f \,\big\|_{L^2(\bbR;d\hatt {\mathbf{\Omega}}_j(\, \cdot \,,x_0))}  \no \\
& = \big\|\hatt f \,\big\|_{L^2(\bbR;d\hatt {\mathbf{\Omega}}_j(\, \cdot \,,x_0))}, \quad \hatt f\in C_0^\infty(\bbR)^m.   \lb{4.55}
\end{align}
From the limiting relation in \eqref{4.53}, one also infers that
\begin{equation}\lb{4.56}
W_j(x_0)U_j(x_0)=I_{L^2(\bbR)^m}.
\end{equation}
To verify that $U_j(x_0)$, $j=1,2$, are onto, and hence unitary, it suffices to prove that $W_j(x_0)$ are injective. Suppose that $\hatt f=(f_0,f_1)^\top\in \ker(W_j(x_0))$.  Let $\hatt f\in L^2(\bbR;\hatt {\mathbf{\Omega}}_j(\, \cdot \,,x_0))$, $\lambda_1, \lambda_2 \in\bbR$,
$\lambda_1<\lambda_2$, and consider
\begin{align}\lb{a4.57}
& (H_j - zI_{L^2(\bbR)^m}) \bigg(\int_{(\lambda_1,\lambda_2]} \big(c_j(\lambda,\, \cdot \,,x_0),s_j(\lambda,\, \cdot \,,x_0)\big) \,
(\lambda - z)^{-1}d\hatt {\mathbf{\Omega}}_j(\lambda,x_0) \, \hatt f(\lambda)\bigg)   \no  \\
& \quad = \int_{(\lambda_1,\lambda_2]}
\big(c_j(\lambda,\, \cdot \,,x_0),s_j(\lambda,\, \cdot \,,x_0)\big) \, d\hatt {\mathbf{\Omega}}_j(\lambda,x_0) \, \hatt f(\lambda), \quad z\in\bbC_+. 
\end{align}
Then,
\begin{align}
& \int_{(\lambda_1,\lambda_2]}  \big(c_j(\lambda,\, \cdot \,,x_0),s_j(\lambda,\, \cdot \,,x_0)\big)\, 
(\lambda - z)^{-1}\, d\hatt {\mathbf{\Omega}}_j(\lambda,x_0) \, \hatt f(\lambda)    \no\\
& \quad = (H_j - zI_{L^2(\bbR)^m})^{-1} \bigg(\int_{(\lambda_1,\lambda_2]}
\big(c_j(\lambda,\, \cdot \,,x_0),s_j(\lambda,\, \cdot \,,x_0)\big) \, d\hatt {\mathbf{\Omega}}_j(\lambda,x_0)
\, \hatt f(\lambda)\bigg),    \no \\
&\hspace*{9.9cm}  z\in\bbC_+.   \lb{a4.58}
\end{align}
Taking $\slim_{\lambda_1\downarrow -\infty, \lambda_2\uparrow \infty}$ in \eqref{a4.58} implies
\begin{equation}\lb{a4.59}
W_j(x_0) \big((\,\cdot -z)^{-1} \hatt f \,\big) = (H_j -zI_{L^2((\bbR)^m})^{-1} W_j(x_0) \hatt f,
\quad z\in\bbC_+.
\end{equation}
Next, suppose that $\hatt f_0 = (f_0,f_1)^{\top} \in \ker(W_j(x_0))$, and take a sequence
$\big\{\hatt f_n\big\}_{n\in\bbN} \subset L^2(\bbR;d\hatt {\mathbf{\Omega}}_j(\, \cdot \,,x_0))$ such that
$\supp\big(\hatt f_n\big)$ is compact for each $n\in\bbN$ and
$\lim_{n\uparrow\infty}\big\|\hatt f_0 - \hatt f_n\big\|_{L^2(\bbR;d\hatt {\mathbf{\Omega}}_j(\, \cdot \, ,x_0))}=0$.  Then, since each $\hatt f_n$ is compactly supported,
\begin{align}\lb{a4.60}
\begin{split}
\big(W_j(x_0) \big((\, \cdot -z)^{-1} \hatt f_n\big)\big)(x)
= \big((H_j - zI_{L^2(\bbR)^m})^{-1} W_j(x_0) \hatt f_n\big)(x),&
\\
x\in \bbR, \; z\in\bbC_+, \; n\in\bbN.&
\end{split}
\end{align}
Consequently, for each $y\in \bbR$ and all $e\in \bbC^m$, 
\begin{align}
& \int_{x_0}^y dx \int_{\bbR} \Big(e,\big(c_j(\lambda,x,x_0),s_j(\lambda,x,x_0)\big) \, (\lambda - z)^{-1}
\, d\hatt {\mathbf{\Omega}}_j(\lambda,x_0) \, \hatt f_n(\lambda)\Big)_{\bbC^m}   \no \\
& \quad =\int_{\bbR} \bigg( \int_{x_0}^y dx \, \big(c_j(\lambda,x,x_0),s_j(\lambda,x,x_0)\big)^* e,
d\hatt {\mathbf{\Omega}}_j(\lambda,x_0) \, (\lambda - z)^{-1} \hatt f_n(\lambda)\bigg)_{\bbC^m} \no \\
& \quad = \int_{x_0}^y dx \, \Big(e, \big((H_j-zI_{L^2(\bbR)^m})^{-1} W_j(x_0) \hatt f_n\big)(x)\Big)_{\bbC^m}.
\lb{a4.61}
\end{align}
One observes that
\begin{equation}\lb{a4.62}
\begin{split}
&\int_{x_0}^{\infty} dx \, \big(c_j(\lambda,x,x_0),s_j(\lambda,x,x_0)\big)^* \chi_{[x_0,y]}(x) e\\
&\quad = (U_j(x_0) \chi_{[x_0,y]} e)(\cdot) \in L^2(\bbR;d\hatt {\mathbf{\Omega}}_j(\, \cdot \, ,x_0)).
\end{split}
\end{equation}
Thus, taking the limit $n\uparrow \infty$ in \eqref{a4.60} yields
\begin{align}
&\lim_{n\uparrow\infty}
\int_{x_0}^y dx \int_{\bbR}  \Big(e, \big(c_j(\lambda,x,x_0),s_j(\lambda,x,x_0)\big)
d\hatt {\mathbf{\Omega}}_j(\lambda,x_0) \, (\lambda - z)^{-1}  \hatt f_n(\lambda)\Big)_{\bbC^m} \no \\
& \quad =  \int_{\bbR}  (\lambda - z)^{-1}  \int_{x_0}^y dx \, \Big(e, \big(c_j(\lambda,x,x_0),s_j(\lambda,x,x_0)\big)
d\hatt {\mathbf{\Omega}}_j(\lambda,x_0) \hatt f_0(\lambda)\Big)_{\bbC^m}   \lb{a4.63} \\
& \quad = \lim_{n\uparrow\infty}
\int_{x_0}^y dx \, \Big(e, \big((H_j-zI_{L^2((\bbR)^m})^{-1} W_j(x_0) \hatt f_n\big)(x)\Big)_{\bbC^m} \no \\
& \quad = \int_{x_0}^y dx \, \big(e, \big((H_j-zI_{L^2(\bbR)^m})^{-1} W_j(x_0) \hatt f_0\big)(x)\big)_{\bbC^m} =0,   \no\\
&\hspace*{4.9cm}y\in \bbR, \; z\in\bbC_+, \; e\in\bbC^m.   \no
\end{align}
Applying the Stieltjes inversion formula to the (finite) complex-valued measure in
the 3rd line of \eqref{a4.63}, given by,
\begin{equation}\lb{a4.64}
\int_{x_0}^y dx \, \Big(e, \big(c_j(\lambda,x,x_0),s_j(\lambda,x,x_0)\big)
\, d\hatt {\mathbf{\Omega}}_j(\lambda,x_0) \, \hatt f_0(\lambda)\Big)_{\bbC^m},
\end{equation}
implies for all $\lambda_1, \lambda_2 \in\bbR$, $\lambda_1 < \lambda_2$, and $e \in \bbC^m$,
\begin{equation}
\int_{(\lambda_1,\lambda_2]}
\int_{x_0}^y dx \, \Big(e, \big(c_j(\lambda,x,x_0),s_j(\lambda,x,x_0)\big)
\, d\hatt {\mathbf{\Omega}}_j(\lambda,x_0) \, \hatt f_0(\lambda)\Big)_{\bbC^m} =0, \quad y\in \bbR.   \lb{a4.65}
\end{equation}
Differentiating \eqref{a4.65} with respect to $y$, noting that $c_j(\lambda,y,x_0)$ and $s_j(\lambda,y,x_0)$ are continuous in $(\lambda,y)\in \bbR^2$, and using the dominated convergence theorem, one obtains
\begin{equation}\lb{4.58}
\int_{(\lambda_1,\lambda_2]}\Big(e, \big(c_j(\lambda,y,x_0),s_j(\lambda,y,x_0)\big)d\hatt {\mathbf{\Omega}}_j(\lambda,x_0)\hatt f_0(\lambda)\Big)_{\bbC^m}=0,\quad y\in \bbR, \; e\in \bbC^m.
\end{equation}
In particular, taking $y=x_0$ in \eqref{4.58} and using \eqref{3.17} and 
\eqref{3.18}, one obtains
\begin{align}
&\int_{(\lambda_1,\lambda_2]}\Big(e_1, \big(c_j(\lambda,x_0,x_0),s_j(\lambda,x_0,x_0)\big)d\hatt {\mathbf{\Omega}}_j(\lambda,x_0)\hatt f_0(\lambda)\Big)_{\bbC^m}\no\\
&\quad= \int_{(\lambda_1,\lambda_2]}\big((e_1,0_m)^{\top}, d\hatt {\mathbf{\Omega}}_j(\lambda,x_0)\hatt f_0(\lambda)\big)_{\bbC^{2m}}=0,\quad e_1\in \bbC^m,\lb{4.59}
\end{align}
where $0_m$ in \eqref{4.59} denotes the zero vector in $\bbC^m$.  Next, applying 
the quasi-derivative $[1,j]$ with respect to $y \in \bbR$ to \eqref{4.58}, yields 
\begin{align}\lb{4.58A}
\begin{split}
\int_{(\lambda_1,\lambda_2]}\Big(e, \big(c_j^{[1,j]}(\lambda,y,x_0),s_j^{[1,j]}(\lambda,y,x_0)\big)d\hatt {\mathbf{\Omega}}_j(\lambda,x_0)\hatt f_0(\lambda)\Big)_{\bbC^m}=0,&  \\
y\in \bbR, \; e\in \bbC^m,&
\end{split}
\end{align}
using the fact that
\begin{align}\lb{4.58a}
\begin{split} 
\int_{(\lambda_1,\lambda_2]}
\Big(e, \big(\phi(y) c_j(\lambda,y,x_0), \phi(y) s_j(\lambda,y,x_0)\big)d\hatt {\mathbf{\Omega}}_j(\lambda,x_0)\hatt f_0(\lambda)\Big)_{\bbC^m}=0,& \\ 
y\in \bbR, \; e\in \bbC^m.&
\end{split}
\end{align}
Subsequently, taking $y=x_0$ in \eqref{4.58A}, once more using \eqref{3.17} and 
\eqref{3.18} yields
\begin{equation}\lb{4.60}
\int_{(\lambda_1,\lambda_2]}\big((0_m,e_2)^{\top}, d\hatt {\mathbf{\Omega}}_j(\lambda,x_0)\hatt f_0(\lambda)\big)_{\bbC^{2m}}=0,\quad e_2\in \bbC^m.
\end{equation}
Taking $e=(e_1,e_2)^{\top}$ with $e_1,e_2\in \bbC^m$ and adding \eqref{4.59} and \eqref{4.60}, one obtains
\begin{equation}\lb{4.61}
\int_{(\lambda_1,\lambda_2]}\big(e, d\hatt {\mathbf{\Omega}}_j(\lambda,x_0)\hatt f_0(\lambda)\big)_{\bbC^{2m}}=0,\quad e\in \bbC^{2m}.
\end{equation}
Since $\lambda_1$ and $\lambda_2$ are arbitrary (apart from $\lambda_1<\lambda_2$), \eqref{4.61} implies
\begin{equation}\lb{4.62}
\hatt f_0(\lambda)=0 \, \text{ for $d\hatt {\mathbf{\Omega}}_j(\, \cdot \,,x_0)$-a.e.\ $\lambda \in \bbR$.}
\end{equation}
\end{proof}

The proofs of Theorems \ref{t4.5} and \ref{t4.6} are adaptation of the proofs of Theorems 2.12 and 2.14  
in \cite{GZ06}. This strategy of proof immediately extends to all continuous or discrete second-order 
problems (such as Sturm--Liouville, Jacobi, and CMV operators) and first-order $2 \times 2$ systems (i.e., Dirac-type operators) with matrix-valued coefficients, see, for instance, \cite{CGZ07}, \cite{CGZ08}. In fact, 
it also extends to the infinite-dimensional case of bounded operator-valued coefficients (for the case of 
Schr\"odinger operators with bounded operator-valued potentials, see \cite{GWZ12}). 

In our final result, we show that the known local Borg--Marchenko results for Dirac-type operators worked out in \cite{CG02} immediately imply local Borg--Marchenko results for generalized Schr\"odinger operators of the form $H_j$, $j=1,2$. For simplicity, we focus on $H_1$ only.

\begin{theorem} \lb{t4.7} 
Suppose Hypothesis \ref{h2.1} holds with $\phi_j$, $j=1,2$, in place of $\phi$.  Let $H_{1,1}$ $($resp., $H_{1,2}$$)$ denote the operator defined in \eqref{3.5} taking $\phi=\phi_1$ $($resp., $\phi=\phi_2$$)$ and denote by $\hatt {\mathbf{M}}_{1,1}(z,x_0)$ $($resp., $\hatt {\mathbf{M}}_{1,2}(z,x_0)$$)$ the corresponding $2m\times 2m$ block Weyl--Titchmarsh matrix as defined in \eqref{4.25}--\eqref{4.29}.  Then,  
\begin{equation}\lb{4.65}
\text{if for some $a>0$, $\phi_1(x)=\phi_2(x)$ for a.e.\ $x\in (x_0-a,x_0+a)$,}
\end{equation}
one obtains
\begin{equation}\lb{4.66}
\big\|\hatt {\mathbf{M}}_{1,1}(z,x_0) - \hatt {\mathbf{M}}_{1,2}(z,x_0) \big\|_{\bbC^{2m\times2m}}\underset{\substack{|z|\rightarrow \infty \\ z\in \rho_{\theta}}}{=}
\Oh\big(z^{1/2} e^{-2\Im(z^{1/2})a} \big),
\end{equation}
along any ray $\rho_{\theta}\subset \bbC$ with $\arg(z)=\theta\in (0,\pi)\cup(\pi,2\pi)$, always choosing the branch of the square root with $\Im(z^{1/2})>0$ for $z\in \bbC\backslash [0,\infty)$.  On the other hand, suppose that for all $\varepsilon>0$,
\begin{equation}\lb{4.66a}
\big\|\hatt {\mathbf{M}}_{1,1}(z,x_0)-\hatt {\mathbf{M}}_{1,2}(z,x_0)\big\|_{\bbC^{2m\times 2m}}\underset{\substack{|z|\rightarrow \infty \\ z\in \rho_{\theta_{\ell}}}}{=}\Oh\big(z^{1/2} e^{-2\Im(z^{1/2})(a-\varepsilon)} \big),\quad \ell=1,2,
\end{equation}
along a ray $\rho_{\theta_1}\subset \bbC$ with $\arg(z)=\theta_1$ and $0<\theta_1<\pi$ and along a ray $\rho_{\theta_2}\subset \bbC$ with $\arg(z)=\theta_2$ and $\pi<\theta_2<2\pi$. If $m>1$, assume in addition that $\phi_j\in L^{\infty}([x_0-a,x_0+a])^{m\times m}$, $j=1,2$.  Then
\begin{equation}\lb{4.66b}
\text{$\phi_1(x)=\phi_2(x)$ for a.e.\ $x\in [x_0-a,x_0+a]$.}
\end{equation}
\end{theorem}
\begin{proof}
We begin by fixing some notation.  Let $\mathbf{M}_j^D(\zeta,x_0,\alpha_0)$, $j=1,2$, denote the $2m\times 2m$ Weyl--Titchmarsh matrix defined by \eqref{2.29} corresponding to the Dirac-type operator $D_j$, $j=1,2$ defined by \eqref{2.5}, that is, the operator (formally) defined as
\begin{equation}\lb{4.68}
D_j=J\frac{d}{dx}-B_j(x), \quad B_j(x)=\begin{pmatrix} 0 & - \phi_j(x) \\ - \phi_j(x) & 0\end{pmatrix}, \; j=1,2.
\end{equation}
As a result of \eqref{4.67}, one estimates
\begin{align}
&\big\|\hatt {\mathbf{M}}_{1,1}(z,x_0) - \hatt {\mathbf{M}}_{1,2}(z,x_0) \big\|_{\bbC^{2m\times2m}}\no\\
&\quad \leq C|\zeta|\big\|\mathbf{M}_1^D(\zeta,x_0,\alpha_0)-\mathbf{M}_2^D(\zeta,x_0,\alpha_0)\big\|_{\bbC^{2m\times 2m}},\lb{4.69}\\
&\hspace*{4.33cm}\zeta^2=z\in \bbC\backslash \bbR,\, |z|\geq 1,\no
\end{align}
for a $z$-independent constant $C>0$.

Evidently, taking $|z|\rightarrow \infty$ along the fixed ray $\rho_{\theta}\subset \bbC$ with $\arg(z)=\theta$ and $\theta\in (0,\pi)\cup (\pi,2\pi)$ implies
\begin{equation}\lb{4.70}
\text{$|\zeta|\rightarrow \infty$, along the fixed ray $\rho_{\theta/2}\subset \bbC_+$ with $0<\arg(\zeta)=\theta/2<\pi$,}
\end{equation}
if $\zeta^2=z$.
The assumption in \eqref{4.65} implies
\begin{equation}\lb{4.71}
\text{$B_1(x)=B_2(x)$ for a.e.\ $x\in (x_0-a,x_0+a)$,}
\end{equation}
with $B_j(\cdot)$, $j=1,2$, as in \eqref{4.68}.  Consequently, \cite[Theorem 5.5]{CG02} implies
\begin{equation}\lb{4.72}
\big\|\mathbf{M}_1^D(\zeta,x_0,\alpha_0)-\mathbf{M}_2^D(\zeta,x_0,\alpha_0)\big\|_{\bbC^{2m\times 2m}}\underset{\substack{|\zeta|\rightarrow \infty \\ \zeta\in \rho_{\theta/2}}}{=}\Oh\big(e^{-2\Im(\zeta)a} \big).
\end{equation}
As $\Re(\zeta)=2\theta^{-1}\Im(\zeta)$ for $\zeta \in \rho_{\theta/2}$, \eqref{4.72} implies
\begin{equation}\lb{4.73}
|\zeta|\big\|\mathbf{M}_1^D(\zeta,x_0,\alpha_0)-\mathbf{M}_2^D(\zeta,x_0,\alpha_0)\big\|_{\bbC^{2m\times 2m}}\underset{\substack{|\zeta|\rightarrow \infty \\ \zeta\in \rho_{\theta/2}}}{=}\Oh\big(\Im(\zeta)e^{-2\Im(\zeta)a} \big).
\end{equation}
Together, \eqref{4.69} and \eqref{4.73} yield \eqref{4.66}.

Assuming \eqref{4.66a}, in order to prove \eqref{4.66b}, it suffices to prove 
\begin{equation}\lb{4.74}
\big\|\mathbf{M}_1^D(\zeta,x_0,\alpha_0)-\mathbf{M}_2^D(\zeta,x_0,\alpha_0)\big\|_{\bbC^{2m\times 2m}}\underset{\substack{|\zeta|\rightarrow \infty \\ \zeta\in \rho_{\theta_{\ell}/2}}}{=}\Oh\big(e^{-2\Im(\zeta)(a-\epsilon)} \big),\quad \epsilon>0,\, \ell=1,2, 
\end{equation}
since $\rho_{\theta_1/2}$ (resp., $\rho_{\theta_2/2})$ is a ray in $\bbC$ with $0<\arg(\zeta)=\theta_1/2<\pi/2$ (resp., $\pi/2<\arg(\zeta)=\theta_2/2<\pi$).
Indeed, by \cite[Theorem 5.5]{CG02}, \eqref{4.74} implies
\begin{equation}\lb{4.75}
\text{$B_1(x)=B_2(x)$ for a.e.\ $x\in [x_0-a,x_0+a]$,}
\end{equation}
which clearly yields \eqref{4.66b}.  In order to prove \eqref{4.74}, let $\epsilon>0$.  Making use of \eqref{4.67}, one writes
\begin{align}
&\big\|\mathbf{M}_1^D(\zeta,x_0,\alpha_0)-\mathbf{M}_2^D(\zeta,x_0,\alpha_0)\big\|_{\bbC^{2m\times 2m}}\no\\
&\quad = \bigg\|\begin{pmatrix} \zeta I_m & 0\\ 0 & I_m\end{pmatrix}\Big[\hatt {\mathbf{M}}_{1,1}(z,x_0) - \hatt {\mathbf{M}}_{1,2}(z,x_0) \Big] \begin{pmatrix} I_m & 0 \\ 0 & \zeta^{-1} I_m \end{pmatrix}\bigg\|_{\bbC^{2m\times 2m}}\lb{4.76}\\
&\quad \leq C|\zeta|\Big\| \hatt {\mathbf{M}}_{1,1}(z,x_0) - \hatt {\mathbf{M}}_{1,2}(z,x_0)\Big\|_{\bbC^{2m\times 2m}}, \quad z=\zeta^2,\, \zeta\in\rho_{\theta_{\ell}/2},\, |\zeta|\geq 1,\, \ell=1,2,\no
\end{align}
for some $\zeta$-independent constant $C>0$.  Choosing $\varepsilon=2\epsilon$, one computes
\begin{align}
&e^{2\Im(\zeta)(a-\epsilon)}\big\|\mathbf{M}_1^D(\zeta,x_0,\alpha_0)-\mathbf{M}_2^D(\zeta,x_0,\alpha_0)\big\|_{\bbC^{2m\times 2m}}\no\\
&\quad \leq C\, |\zeta |e^{2\Im(\zeta)(a-\epsilon)}\Big\| \hatt {\mathbf{M}}_{1,1}(\zeta^2,x_0) - \hatt {\mathbf{M}}_{1,2}(\zeta^2,x_0)\Big\|_{\bbC^{2m\times 2m}}\lb{4.77}\\
&\quad \leq \widetilde{C}\, \Im(\zeta)e^{2\Im(\zeta)(a-\epsilon)}e^{-2\Im(\zeta)(a-\varepsilon)}\lb{4.78}\\
&\quad = \widetilde{C}\, \Im(\zeta)e^{-2\Im(\zeta)\epsilon},\quad \zeta\in \rho_{\theta_{\ell}/2},\, |\zeta |\gg 1,\, \ell=1,2,\lb{4.79}
\end{align}
where $\tilde C>0$ is an appropriate $\zeta$-independent constant.
The estimate in \eqref{4.77} makes use of \eqref{4.76} and the estimate in \eqref{4.78} uses the assumption in \eqref{4.66a} which is applicable since $\zeta\in \rho_{\theta_{\ell}/2}$, $|\zeta |\gg 1$ implies $z=\zeta^2\in \rho_{\theta_{\ell}}$, $|z|\gg 1$.  The function of $\zeta$ in \eqref{4.79} is bounded as $|\zeta |\rightarrow \infty$, $\zeta \in \rho_{\theta_{\ell}/2}$, $\ell=1,2$; hence, \eqref{4.79} implies 
\begin{equation}\lb{4.80}
\big\|\mathbf{M}_1^D(\zeta,x_0,\alpha_0)-\mathbf{M}_2^D(\zeta,x_0,\alpha_0)\big\|_{\bbC^{2m\times 2m}}\underset{\substack{|\zeta|\rightarrow \infty \\ \zeta\in \rho_{\theta_{\ell}/2}}}{=}\Oh\big(e^{-2\Im(\zeta)(a-\epsilon)} \big),\quad \ell=1,2.
\end{equation}
Since $\epsilon>0$ was arbitrary, \eqref{4.74} is established.
\end{proof}

The reason for the additional assumption $\phi_j\in L^{\infty}([x_0-a,x_0+a])^{m\times m}$, $j=1,2$, in the case $m>1$, is due to a technical issue as explained in detail in \cite[Remark\ 5.4]{CG02}. It should 
be noted in this connection that for matrix-valued Schr\"odinger operators $H_{1,j}$, $j=1,2$, with the standard hypothesis on potentials, that is, $V_j = V_j^* \in L^1_{\loc}(\bbR)^{m\times m}$, the local 
Borg--Marchenko-type results in Theorem \ref{t4.7} were derived in \cite{GKM02} (see also \cite{GS00}). 
In this case no analog of the local boundedness assumptions on $\phi_j$ were necessary.  

Finally, in the context of Theorem \ref{t4.7}, we would also like to point out that an alternative approach to 
local Borg--Marchenko-type uniqueness results for Dirac-type operators, including a procedure for 
recovering the potential coefficient, was developed by Sakhnovich \cite{Sa02} (see also \cite{Sa06} and 
\cite[Sects.\ 2.2.3, 4.2.2]{SSR13}).

\appendix
\section{Supersymmetric Dirac-Type Operators in a Nutshell} \lb{sA}
\renewcommand{\theequation}{A.\arabic{equation}}
\renewcommand{\thetheorem}{A.\arabic{theorem}}
\setcounter{theorem}{0} \setcounter{equation}{0}

In this appendix we briefly summarize some results on supersymmetric 
Dirac-type operators and commutation methods due to \cite{De78}, 
\cite{GSS91}, \cite{Th88}, and \cite[Ch.\ 5]{Th92} (see also \cite{HKM00}). 

The standing assumption in this appendix will be the following.

\begin{hypothesis} \lb{hA.1}
Let $\cH_j$, $j=1,2$, be separable complex Hilbert spaces and 
\begin{equation}
T: \cH_1 \supseteq \dom(T) \to \cH_2    \lb{A.1}
\end{equation} 
be a densely defined, closed, linear operator. 
\end{hypothesis}

We define the self-adjoint Dirac-type operator in $\cH_1 \oplus \cH_2$ by 
\begin{equation}
Q = \begin{pmatrix} 0 & T^* \\ T & 0 \end{pmatrix}, \quad \dom(Q) = \dom(T) \oplus \dom(T^*).     
\lb{A.2}
\end{equation}
Operators of the type $Q$ play a role in supersymmetric quantum mechanics (see, e.g., the extensive list of references in \cite{BGGSS87}). Then,
\begin{equation}
Q^2 = \begin{pmatrix} T^* T & 0 \\ 0 & T T^* \end{pmatrix}     \lb{A.3}
\end{equation}
and for notational purposes we also introduce
\begin{equation}
H_1 = T^* T \, \text{ in } \, \cH_1, \quad H_2 = T T^* \, \text{ in } \, \cH_2.    \lb{A.4}
\end{equation}
In the following, we also need the polar decomposition of $T$ and $T^*$, that is, the representations
\begin{align}
T& = V_T |T| = |T^*| V_T = V_T T^* V_T  \, \text{ on } \, \dom(T) = \dom(|T|),  
\label{A.4a} \\
T^*& = V_{T^*} |T^*| = |T|V_{T^*} = V_{T^*} T V_{T^*}   \, \text{ on } \, 
\dom(T^*) = \dom(|T^*|),   \label{A.4b} \\
|T|& = V_{T^*} T=T^* V_T = V_{T^*} |T^*| V_T \, \text{ on } \, \dom(|T|),    \label{A.4c} \\
 |T^*|& = V_T T^* = T V_{T^*} = V_T |T| V_{T^*}  \, \text{ on } \, \dom(|T^*|),   
 \label{A.4d}
\end{align}
where
\begin{align}
& |T| = (T^* T)^{1/2}, \quad |T^*| = (T T^*)^{1/2},  \quad
V_{T^*} = (V_T)^*,     \lb{A.6a} \\
& V_{T^*} V_T = P_{\ol{{\ran}(|T|)}}=P_{\ol{{\ran}(T^*)}} \, ,  \quad
V_T V_{T^*} = P_{\ol{{\ran}(|T^*|)}}=P_{\ol{{\ran}(T)}} \, .      \label{A.4e}
\end{align}
In particular, $V_T$ is a partial isometry with initial set $\ol{{\ran}(|T|)}$
and final set $\ol{{\ran}(T)}$ and hence $V_{T^*}$ is a partial isometry with initial 
set $\ol{\ran(|T^*|)}$ and final set $\ol{\ran(T^*)}$. In addition, 
\begin{equation}
V_T = \begin{cases} \ol{T (T^* T)^{-1/2}} = \ol{(T T^*)^{-1/2}T} & \text{on }  (\ker (T))^{\bot},  \\
0 & \text{on }  \ker (T).  \end{cases}     \lb{A.7}
\end{equation} 

Next, we collect some properties relating $H_1$ and $H_2$.

\begin{theorem} [\cite{De78}] \lb{tA.2}  
Assume Hypothesis \ref{hA.1} and let $\phi$ be a bounded Borel measurable 
function on $\bbR$. \\ 
$(i)$ One has
\begin{align}
& \ker(T) = \ker(H_1) = (\ran(T^*))^{\bot}, \quad \ker(T^*) = \ker(H_2) = (\ran(T))^{\bot},  \lb{A.8} \\
& V_T H_1^{n/2} = H_2^{n/2} V_T, \; n\in\bbN, \quad 
V_T \phi(H_1) = \phi(H_2) V_T.  \lb{A.9} 
\end{align}
$(ii)$ $H_1$ and $H_2$ are essentially isospectral, that is, 
\begin{equation}
\sigma(H_1)\backslash\{0\} = \sigma(H_2)\backslash\{0\},    \lb{A.10}
\end{equation}
in fact, 
\begin{equation}
T^* T [I_{\cH_1} - P_{\ker(T)}] \, \text{ is unitarily equivalent to } \, T T^* [I_{\cH_2} - P_{\ker(T^*)}]. 
\lb{A.10a}
\end{equation} 
In addition,
\begin{align}
& f\in \dom(H_1) \, \text{ and } \, H_1 f = \lambda^2 f, \; \lambda \neq 0,   \no \\
& \quad \text{implies }  \,  T f \in \dom(H_2) \, \text{ and } \, H_2(T f) = \lambda^2 (T f),    \lb{A.11} \\
& g\in \dom(H_2)\, \text{ and } \, H_2 \, g = \mu^2 g, \; \mu \neq 0,     \no \\
& \quad \text{implies }  \, T^* g \in \dom(H_1)\, \text{ and } \, H_1(T^* g) = \mu^2 (T^* g),    \lb{A.12} 
\end{align}
with multiplicities of eigenvalues preserved. \\
$(iii)$ One has for $z \in \rho(H_1) \cap \rho(H_2)$,
\begin{align}
& I_{\cH_2} + z (H_2 - z I_{\cH_2})^{-1} \supseteq T (H_1 - z I_{\cH_1})^{-1} T^*,    \lb{A.13} \\
& I_{\cH_1} + z (H_1 - z I_{\cH_1})^{-1} \supseteq T^* (H_2 - z I_{\cH_2})^{-1} T,    \lb{A.14}
\end{align}
and 
\begin{align}
& T^* \phi(H_2) \supseteq \phi(H_1) T^*, \quad 
T \phi(H_1) \supseteq \phi(H_2) T,   \lb {A.15} \\
& V_{T^*} \phi(H_2) \supseteq \phi(H_1) V_{T^*}, \quad 
V_T \phi(H_1) \supseteq \phi(H_2) V_T.   \lb {A.15a}
\end{align}
\end{theorem}

As noted by E.\ Nelson (unpublished), Theorem \ref{tA.2} follows from the spectral theorem and the 
elementary identities, 
\begin{align}
& Q = V_Q |Q| = |Q| V_Q,    \lb{A.16} \\
& \ker(Q) = \ker(|Q|) = \ker (Q^2) = (\ran(Q))^{\bot} 
= \ker (T) \oplus \ker (T^*),    \lb{A.17}  \\
& I_{\cH_1 \oplus \cH_2} + z (Q^2 - z I_{\cH_1 \oplus \cH_2})^{-1} 
= Q^2 (Q^2 -z I_{\cH_1 \oplus \cH_2})^{-1} \supseteq Q (Q^2 -z I_{\cH_1 \oplus \cH_2})^{-1} Q,  \no \\
& \hspace*{9.3cm}   z \in \rho(Q^2),    \lb{A.18} \\
& Q \phi(Q^2) \supseteq \phi(Q^2) Q,   \lb{A.19}
\end{align}
where
\begin{equation}
V_Q = \begin{pmatrix} 0 & (V_T)^* \\ V_T & 0 \end{pmatrix} 
= \begin{pmatrix} 0 & V_{T^*} \\ V_T & 0 \end{pmatrix}.   \lb{A.20}
\end{equation}

In particular,
\begin{equation}
\ker(Q) = \ker(T) \oplus \ker(T^*), \quad  
P_{\ker(Q)} = \begin{pmatrix} P_{\ker(T)} & 0 \\ 0 & P_{\ker(T^*)} \end{pmatrix},    \lb{A.21}
\end{equation}
and we also recall that
\begin{equation}
\mathfrak{S}_3 Q \mathfrak{S}_3 = - Q, \quad \mathfrak{S}_3 = \begin{pmatrix} I_{\cH_1} & 0 \\ 0 & - I_{\cH_2} 
\end{pmatrix},     \lb{A.22}
\end{equation}
that is, $Q$ and $-Q$ are unitarily equivalent. (For more details on Nelson's 
trick see also \cite[Sect.\ 8.4]{Te09}, \cite[Subsect.\ 5.2.3]{Th92}.) 
We also note that
\begin{equation}
\psi(|Q|) = \begin{pmatrix} \psi(|T|) & 0 \\ 0 & \psi(|T^*|) \end{pmatrix}    \lb{A.22a} 
\end{equation}
for Borel measurable functions $\psi$ on $\bbR$, and 
\begin{equation}
\ol{[Q |Q|^{-1}]} = \begin{pmatrix} 0 & (V_T)^*\\ V_T & 0 \end{pmatrix} = V_Q 
\, \text{ if } \, \ker(Q) = \{0\}.   \lb{A.22b} 
\end{equation}

Finally, we recall the following relationships between $Q$ and $H_j$, $j=1,2$.

\begin{theorem} [\cite{BGGSS87}, \cite{Th88}] \lb{tA.3}
Assume Hypothesis \ref{hA.1}. \\
$(i)$ Introducing the unitary operator $U$ on $(\ker(Q))^{\bot}$ by
\begin{equation}
U = 2^{-1/2} \begin{pmatrix} I_{\cH_1} & (V_T)^* \\ -V_T & I_{\cH_2} \end{pmatrix} 
\, \text{ on } \,  (\ker(Q))^{\bot},     \lb{A.23}
\end{equation}
one infers that
\begin{equation}
U Q U^{-1} = \begin{pmatrix}  |A| & 0 \\ 0 & - |A^*| \end{pmatrix} 
\, \text{ on } \,  (\ker(Q))^{\bot}.     \lb{A.24}
\end{equation}
$(ii)$ One has
\begin{align}
\begin{split}
(Q - \zeta I_{\cH_1 \oplus \cH_2})^{-1} = \begin{pmatrix} \zeta (H_1 - \zeta^2 I_{\cH_1})^{-1} 
& T^* (H_2 - \zeta^2 I_{\cH_2})^{-1}  \\  T (H_1 - \zeta^2 I_{\cH_1})^{-1}  & 
\zeta (H_2 - \zeta^2 I_{\cH_2})^{-1}  \end{pmatrix},&    \\
\zeta^2 \in \rho(H_1) \cap \rho(H_2).&   \lb{A.25}
\end{split}
\end{align}
$(iii)$ In addition, 
\begin{align}
\begin{split} 
& \begin{pmatrix} f_1 \\ f_2 \end{pmatrix} \in \dom(Q) \, \text{ and } \, 
Q \begin{pmatrix} f_1 \\ f_2 \end{pmatrix} = \eta \begin{pmatrix} f_1 \\ f_2 \end{pmatrix}, \; \eta \neq 0,  
 \\
& \quad \text{ implies } \, f_j \in \dom (H_j) \, \text{ and } \, H_j f_j = \eta^2 f_j, \; j=1,2.    \lb{A.26}
\end{split} 
\end{align}
Conversely,
\begin{align}
\begin{split} 
& f \in \dom(H_1) \, \text{ and } H_1 f = \lambda^2 f, \; \lambda \neq 0, \\
& \quad \text{implies } \, \begin{pmatrix} f \\ \lambda^{-1} T f \end{pmatrix} \in \dom(Q) \, \text{ and } \, 
Q \begin{pmatrix} f \\ \lambda^{-1} T f \end{pmatrix} 
= \lambda \begin{pmatrix} f \\ \lambda^{-1} T f \end{pmatrix}.  
\end{split} 
\end{align}
Similarly,
\begin{align}
\begin{split} 
& g \in \dom(H_2) \, \text{ and } H_2 \, g = \mu^2 g, \; \mu \neq 0, \\
& \quad \text{implies } \, \begin{pmatrix} \mu^{-1} T^* g \\ g \end{pmatrix} \in \dom(Q) \, \text{ and } \, 
Q \begin{pmatrix} \mu^{-1} T^* g \\ g \end{pmatrix} 
= \mu \begin{pmatrix} \mu^{-1} T^* g \\ g \end{pmatrix}.  
\end{split} 
\end{align}
\end{theorem}

\medskip

\noindent {\bf Acknowledgments.} We are indebted to Rostyk Hryniv and Alexander 
Sakhnovich for very helpful discussions. 


\end{document}